\documentclass[letter,journal,twoside,final]{IEEEtran}  
\IEEEoverridecommandlockouts       
\usepackage{amsthm} 

\usepackage{cite}
\usepackage{amsmath,amssymb,amsfonts}
\usepackage{algorithm,algorithmic,setspace}
\usepackage{graphicx}
\usepackage{textcomp}
\usepackage{soul}
\usepackage{epstopdf} 
\usepackage{graphics} 
\usepackage{subfig}
\usepackage{tikz} 
\usepackage{epstopdf} 
\usepackage{url}
\usepackage{hyperref}
\hypersetup{
	colorlinks=true,
	linkcolor=blue,
	filecolor=blue,
	urlcolor=blue,
	citecolor = blue,
}
\newtheorem{theorem}{Theorem}
\newtheorem{lemma}{Lemma}
\newtheorem{proposition}{Proposition}
\newtheorem{definition}{Definition}
\newtheorem{assumption}{Assumption}

\newtheorem{remark}{Remark}

\DeclareMathOperator{\tr}{tr}  
\DeclareMathOperator{\diag}{diag}
\DeclareMathOperator{\blkdiag}{blkdiag}
\DeclareMathOperator{\Cov}{Cov}
\newcommand{\T}{^{\top}}  
\newcommand{\mycomment}[1]{{\hfill\ttfamily\footnotesize\fontdimen2\font=0.3em\textcolor{blue}{/*#1*/}}}  

\newcommand{\ie}{\textit{i.e.}}  
\allowdisplaybreaks
\def\BibTeX{{\rm B\kern-.05em{\sc i\kern-.025em b}\kern-.08em
		T\kern-.1667em\lower.7ex\hbox{E}\kern-.125emX}}

\title{
	Nonlinear Observers Design for Vision-Aided Inertial Navigation Systems
} 
\author{Miaomiao Wang,  Soulaimane Berkane,  and Abdelhamid Tayebi  
	\thanks{This work was supported by the National Sciences and Engineering Research Council of Canada (NSERC), under the grants NSERC-DG RGPIN 2020-06270 and NSERC-DG RGPIN-2020-04759.  
		This paper was presented in part at the 58th IEEE Conference on Decision and Control, Nice, France, December 2019 \cite{wang2019nonlinear}.} 
	\thanks{M. Wang is with the Department of Electrical and Computer Engineering,
		Western University, London, ON N6A 3K7, Canada (e-mail: mwang448@uwo.ca).}%
	\thanks{S. Berkane is with the Department of Department of Computer Science and Engineering,   University of Quebec in Outaouais, QC J8X 3X7, Canada (e-mail: soulaimane.berkane@uqo.ca).}%
	\thanks{A. Tayebi is with the Department of Electrical and Computer Engineering, 
		Western University, London, ON N6A 3K7, Canada, and also with the Department of Electrical Engineering, Lakehead University, Thunder Bay, ON P7B 5E1, Canada
		(e-mail:  atayebi@lakeheadu.ca).} 
}%

\begin{document} 
	
	\maketitle 
	
	\begin{abstract}
		This paper deals with the simultaneous estimation of the attitude, position and linear velocity for vision-aided inertial navigation systems. We propose a nonlinear observer on $SO(3)\times \mathbb{R}^{15}$ relying on body-frame acceleration, angular velocity and (stereo or monocular) bearing measurements of some landmarks that are constant and known in the inertial frame. Unlike the existing local Kalman-type observers, our proposed nonlinear observer guarantees almost global asymptotic stability and local exponential stability. A detailed uniform observability analysis has been conducted and sufficient conditions are derived. Moreover, a hybrid version of the proposed observer is provided to handle the intermittent nature of the measurements in practical applications. Simulation and experimental results are provided to illustrate the effectiveness of the proposed state observer.
		
	\end{abstract}
	
	
	\section{Introduction}
	The design of reliable state observers for the simultaneous estimation of the attitude (orientation), position, and linear velocity for inertial navigation systems (INSs) is crucial in many robotics and aerospace applications. Visual sensors, which provide a rich information about the environment, are becoming ubiquitous in a wide range of applications from mobile augmented reality to autonomous vehicles' navigation. A vision system together with an inertial measurement
	unit (IMU) forms a vision-aided INS (or visual-inertial navigation system). Depending on the camera type, vision-aided INSs can be categorized as monocular vision-aided INS (single camera) or stereo vision-aided INS (two cameras in a stereo setup). Vision-aided INS is widely used in robotics, for instance, visual-inertial odometry (VIO)  and Simultaneous
	Localization and Mapping (SLAM) \cite{mourikis2007multi,li2013high,mur2015orb,qin2018vins}. Unlike the VIO  and SLAM, where the landmarks/features are unknown, the estimation problem at hand, referred to as vision-aided INS, assumes known landmarks in the inertial frame. Visual measurements of these landmarks together with IMU measurements are used to simultaneously estimate  the attitude, position and linear velocity.  The proposed estimation scheme, endowed with almost global asymptotic stability\footnote{The equilibrium point is stable and asymptotically attractive from almost all initial conditions except from a set of Lebesgue measure zero.} (AGAS) guarantees, is able to handle different types of measurements, including monocular-bearing and stereo-bearing measurements.
	
	
	\subsection{Motivation and Prior Literature}
	The position and linear velocity of a rigid body can be obtained, for instance, from a Global Positioning System (GPS), while its attitude can be estimated from a set of body-frame measurements of some known inertial vectors \cite{mahony2008nonlinear,bonnabel2009non,hua2013implementation}. 
	Typically, low-cost IMU-based estimation techniques assume that the accelerometer provides body-frame measurements of the gravity vector, which is not true in applications involving non-negligible accelerations. One solution to this problem consists in using the so-called velocity-aided attitude observers \cite{hua2010attitude,roberts2011,berkane2017attitude}, which make use of the linear velocity and IMU measurements to estimate the attitude. Instead of assuming the linear velocity available (as in the velocity-aided observers), GPS-aided navigation observers use IMU and GPS information to simultaneously estimate the attitude, position and linear velocity; see, for instance, \cite{bryne2017nonlinear} and references therein.
	These estimation schemes, however, are not suitable for implementations in GPS-denied environments (\textit{e.g.,} indoor applications). 

	One alternative approach, that is widely adopted in GPS-denied environments, is the vision-aided INS.		
	Most of the existing estimation schemes for vision-aided INSs in the literature are of Kalman-type such as the Extended Kalman Filter (EKF) and the Unscented Kalman Filter (UKF) \cite{mourikis2007multi,li2013high}. Different versions of
	these algorithms have been proposed in the literature depending on the type of vision-based measurements used. Due to the nonlinearity of the vision-aided INS kinematics, these Kalman-type filters, relying on local linearizations, do not provide strong stability guarantees. Recently, several nonlinear observers using three-dimensional (3D) landmark position measurements, have been proposed in the literature \cite{barrau2017invariant,hua2018riccati,wang2020hybrid}. In contrast with the standard EKF and its variants, an invariant EKF (IEKF), with provable local stability guarantees, has been proposed in \cite{barrau2017invariant}. A local Riccati-based nonlinear observer, inspired from  \cite{hamel2018riccati}, has been proposed in\cite{hua2018riccati}. Motivated by the work in \cite{berkane2017hybrid,wang2018hybrid}, hybrid nonlinear observers, with global exponential stability guarantees, have been proposed in \cite{wang2020hybrid}.
	
	It is clear that vision systems do not directly provide 3D landmark position measurements. In fact, they can be obtained from the images of a stereo-vision system using additional algorithms \cite{hartley2003multiple}.
	In \cite{wang2019nonlinear}, we developed nonlinear observers for attitude, position, linear velocity and gravity vector estimation, using direct stereo-bearing measurements. A local Riccati observer for attitude, position, linear velocity and accelerometer-bias estimation, with monocular-bearing measurements, has been proposed in \cite{marco2020position}. On the other hand, in practical applications, the used sensors may have different sampling rates. For instance, the sampling rates of the vision sensors are much lower than those of the IMU, which is due to the hardware of the vision sensors and the heavy image processing computations. In this situation, one should be careful as the stability results derived for continuous-time observers are not preserved when the measurements are intermittent. To address this problem, hybrid nonlinear observers have been considered in \cite{ferrante2016state,berkane2019attitude,wang2020nonlinear}. 
	
	\subsection{Contributions and Organization of the Paper}
	In the present paper, we propose an AGAS nonlinear observer for the simultaneous estimation of the attitude, position and linear velocity using body-frame accelerometer and gyro measurements as well as monocular or stereo bearing measurements. Note that, due to the motion space topology,  AGAS is the strongest result one can achieve with smooth time-invariant observers. We also provide a hybrid version of the proposed observer that takes into account the sampled and intermittent nature of the visual measurements for practical implementation purposes. Some highlights of the contributions of this paper are as follows:
	\begin{itemize}
		\item [1)] To the best of our knowledge, this work is the first to achieve AGAS results  for vision-aided INS  with  bearing measurements. Note that the Riccati observers in \cite{hamel2018riccati,marco2020position}  provide only local stability guarantees.
		The key idea that allowed to achieve this strong stability result is the introduction of some auxiliary state variables allowing to appropriately design some output-driven signals leading to linear time-varying dynamics for the translational vector state estimation errors. 
		
		\item [2)] The proposed observer uses generic vision-based measurements, including stereo-bearing and monocular-bearing measurements. This is a distinct feature from the existing nonlinear observers which are tailored to a specific type of vision-based measurements \cite{wang2019nonlinear,barrau2017invariant,hua2018riccati,marco2020position,wang2020hybrid,wang2020nonlinear}. This is achieved by introducing some auxiliary basis vectors in the estimation procedure, allowing to derive a generic (possibly time-varying) innovation term for the translational state estimation that captures the different types of vision-based measurements used in this work. 		
		
		\item[3)] A detailed uniform observably analysis has been carried out for the two types of visual landmark information (\text{i.e.,} stereo-bearing and monocular-bearing measurements). Sufficient conditions on the number and location of the landmarks as well as the motion of the vehicle, have been derived. The most challenging analysis was the one related to the monocular vision system which required a tedious proof. 
		
		\item [4)] In practice, the sampling rates of the visual measurements are much lower than those of the IMU measurements (which can be assumed continuous). In this context, we propose a hybrid version of our nonlinear observer to handle the discrete and intermittent nature of the visual measurements, which has been experimentally validated using the EuRoc dataset \cite{Burri25012016}.
		
	\end{itemize}
	
	The rest of this paper is organized as follows. After some  preliminaries  
	in Section \ref{sec:preliminary}, we formulate our estimation problem in Section \ref{sec:III}. 
	Section \ref{sec:IV} is devoted to the design of a generic nonlinear observer for vision-aided INS using different types of vision-based  measurements with stability and observability analysis. A hybrid version of the proposed observer, taking into account the discrete and intermittent nature of the visual measurements, is presented in Section \ref{sec:V}. 
	Simulation and experimental results are presented in Sections \ref{sec:simulation} and \ref{sec:experimental}, respectively.

	\section{Preliminary Material}\label{sec:preliminary}
	\subsection{Notations and Definitions}
	The sets of real, non-negative real, natural numbers and nonzero natural numbers are denoted by $\mathbb{R}$, $\mathbb{R}_{\geq 0}$, $\mathbb{N}$ and $\mathbb{N}_{>0}$, respectively. We denote by $\mathbb{R}^n$ the $n$-dimensional Euclidean space, and by $\mathbb{S}^n$ the set of unit vectors in $\mathbb{R}^{n+1}$. The Euclidean norm of a vector $x\in \mathbb{R}^n$ is denoted by $\|x\|$, and the Frobenius norm of a matrix $X\in \mathbb{R}^{n\times m}$ is denoted by $\|X\|_F = \sqrt{\tr(X\T X)}$. The $n$-by-$n$ identity and zeros matrices are denoted by $I_n$ and $0_n$, respectively. For a given matrix $A\in \mathbb{R}^{n\times n}$, we define $\lambda(A)$ as the set of all eigenvalues of $A$, and  $\mathcal{E}(A)$ as the set of all unit-eigenvectors of $A$. The minimum and maximum eigenvalues of $A$ are, respectively, denoted by $\lambda_{\min}^A$ and $\lambda_{\max}^A$. By $\blkdiag(\cdot)$, we denote the block diagonal matrix.  Let $e_i$ denote the $i$-th basis vector of $\mathbb{R}^n$ which represents the $i$-th column of the identity matrix $I_n$.  
	
	The Special Orthogonal group of order three, denoted by $SO(3)$, is defined as 
	$SO(3):=\{R\in \mathbb{R}^{3}, RR\T = R\T R=I_3, \det(R)=+1\}.$ 
	The \textit{Lie algebra} of $SO(3)$ is given by 
	$\mathfrak{so}(3):=\{\Omega\in \mathbb{R}^{3}: \Omega = -\Omega\T\}.$ 
	Let $\times$ be the vector cross-product on $\mathbb{R}^3$ and define the map $ (\cdot) ^\times: \mathbb{R}^3 \to \mathfrak{so}(3)$ such that $x\times y = x^\times y$, for any $x,y \in \mathbb{R}^3$. Let $\text{vec}: \mathfrak{so}(3) \to \mathbb{R}^3$ be the inverse isomorphism of the map $(\cdot)^\times$, such that $\text{vec}(\omega^\times) = \omega$ for all $\omega \in \mathbb{R}^3$. For a matrix $A\in \mathbb{R}^{3\times 3}$, we denote by $\mathbb{P}_a: \mathbb{R}^{3\times 3} \to \mathfrak{so}(3)$ the anti-symmetric projection of $A$  such that $\mathbb{P}_a(A) := (A-A\T)/2$. Define the composition map $\psi_a: =\text{vec} \circ \mathbb{P}_a $  such that, for a matrix $A=[a_{ij}] \in \mathbb{R}^{3\times 3}$, one has $\psi_a(A) = \frac{1}{2}[a_{32}-a_{23}, a_{13}-a_{31}, a_{21}-a_{12}]\T$.  
	For any $R\in SO(3)$, we define $|R|_I\in [0,1]$ as the normalized Euclidean distance on $SO(3)$ with respect to the identity $I_3$, which is given by $|R|_I^2 = \tr(I_3-R)/4$. 
	We introduce the following important orthogonal projection operator: $\pi: \mathbb{S}^2\to \mathbb{R}^{3\times 3}$ that will be used throughout this paper:
	\begin{align}
		\pi(x) = I_3-xx\T, \quad x\in \mathbb{S}^2  \label{eqn:pi_x}.
	\end{align}
	Note that $\pi(x)$ is an orthogonal projection matrix which geometrically projects any vector in $\mathbb{R}^3$ onto the plane orthogonal to vector $x$. Moreover, one verifies that $\pi(x)$ is bounded and positive semi-definite, $\pi(x) y = 0_{3\times 1}$ if $x,y$ are collinear, and $R\pi(x)R\T = \pi(Rx)$ for any $R\in SO(3),x\in \mathbb{S}^2$. For the sake of simplicity, the argument of the time-dependent signals is omitted unless otherwise required for the sake of clarity.
	
	Consider  $A(t)\in \mathbb{R}^{n \times n}$ and $ C(t)\in \mathbb{R}^{m \times n}$  as matrix-valued functions of time $t$, and suppose that $A(t)$ and $C(t)$ are continuous and bounded on $ \mathbb{R}_{\geq 0}$. The following definition formulates the well-known uniform observability condition in terms of the observability Gramian matrix.
	\begin{definition} 
		The pair $(A(t),C(t))$  is uniformly observable if there exist constants $\delta, \mu>0$ such that
		\begin{align}
			W_o(t,t+\delta) & :=\frac{1}{\delta} \int_{t}^{t+\delta} \Phi\T(\tau,t) C\T(\tau)  C(\tau) \Phi(\tau,t) d\tau  \nonumber \\
			&\geq \mu I_n, \quad \forall t\geq 0   \label{eqn:gramiancondition}
		\end{align}
		where $\Phi(\tau,t)$ is the transition matrix associated to $A(t)$ such that $\frac{d}{dt}\Phi(t,\tau)=A(t)\Phi(t,\tau)$ and $\Phi(t,t) = I_n$. 
	\end{definition}

	\section{Problem Formulation} \label{sec:III}
	\subsection{Kinematic Model}
	Let $\{\mathcal{I}\}$ be an inertial frame and $\{\mathcal{B}\}$ be a body-fixed frame attached to the center of mass of a rigid body. Let the rotation matrix  $R\in SO(3)$ be the attitude of the frame $\{\mathcal{B}\}$ with respect to the frame $\{\mathcal{I}\}$. Let the vectors $p\in \mathbb{R}^3$ and $v\in \mathbb{R}^3$ denote the position and linear velocity of the rigid body expressed in frame $\{\mathcal{I}\}$, respectively.
	The kinematic equations of a rigid body navigating in 3D space are given by:
	\begin{subequations}\label{eqn:INS_system}
		\begin{align}
			\dot{R} & = R \omega^\times  \label{eqn:R}\\
			\dot{p} & = v  \label{eqn:p}\\
			\dot{v} & =   g  + Ra  \label{eqn:v}
		\end{align}
	\end{subequations}
	where $g\in \mathbb{R}^3$ denotes the gravity vector in frame $\{\mathcal{I}\}$, $\omega\in \mathbb{R}^3$ denotes the angular velocity of $\{\mathcal{B}\}$ with respect to $\{\mathcal{I}\}$ expressed in frame $\{\mathcal{B}\}$, and $a\in \mathbb{R}^3$ denotes the ``apparent acceleration" capturing all non-gravitational forces applied to the rigid body expressed in  frame $\{\mathcal{B}\}$.  
	\begin{assumption}\label{assum:IMU}
		The measurements of  the angular velocity $\omega(t)$ and the acceleration $a(t)$ are continuous and bounded.
	\end{assumption}
	
	\subsection{Estimation Problem for Vision-Aided INS}\label{subsection:models}
	This work focuses on the problem of attitude, position and linear velocity estimation for INS using the body-frame acceleration and angular velocity measurements, as well as vision-based body-frame position information of a family of $N\in \mathbb{N}_{>0}$ landmarks that are constant and known in the inertial frame.  Let $p_i$ denote the (constant and known) position of the $i$-th landmark in frame $\{\mathcal{I}\}$, and $p_i^{\mathcal{B}}:= R\T(p_i-p)$ denote the position of the $i$-th landmark in frame $\{\mathcal{B}\}$. In the following, we detail the two measurement models considered in this work (see Fig. \ref{fig:visionsystems}).
	\begin{itemize}		
		\item [1)] \textit{Stereo-bearing measurements:} Let the pairs $(R_{c1}, p_{c1})$ and $(R_{c2}, p_{c2})$ denote the homogeneous transformation from the body-fixed frame $\{\mathcal{B}\}$ to the right camera frame $\{\mathcal{C}_1\}$ and left camera frame $\{\mathcal{C}_2\}$, respectively. Then, the model of the stereo-bearing vectors  of the $i$-th landmark   in frame $\{\mathcal{C}_s\}, s\in\{1,2\}$  are given as 
		\begin{align}
			y_{i}^s :=  \frac{p_i^{\mathcal{C}_s}}{\|p_i^{\mathcal{C}_s}\|} =\frac{R_{cs}\T (p_i^{\mathcal{B}}-p_{cs})}{\| p_i^{\mathcal{B}}-p_{cs}\|}, ~     i\in\{ 1,2 ,\dots,N\} \label{eqn:stereo-bearing}
		\end{align}
		where   $p_i^{\mathcal{C}_s} = R_{cs}\T (p_i^{\mathcal{B}}-p_{cs})$ denotes the   coordinates  of the  $i$-th landmark expressed in frame $\{\mathcal{C}_s\}$ for $s\in\{1,2\}$.
		\item [2)] \textit{Monocular-bearing measurements:} Let the pair $(R_{c}, p_{c})$  denote the homogeneous transformation from the body-fixed frame $\{\mathcal{B}\}$ to the  camera frame $\{\mathcal{C}\}$.  Then, the model of the monocular-bearing vector of the $i$-th landmark  expressed in frame  $\{\mathcal{C}\}$ is given as  
		\begin{align}
			y_i  :=\frac{p_i^{\mathcal{C}} }{\|p_i^{\mathcal{C}} \|} =\frac{R_{c}\T (p_i^{\mathcal{B}}-p_{c})}{\| p_i^{\mathcal{B}}-p_{c}\|}, ~   i\in\{ 1,2 ,\dots,N\} \label{eqn:monocular-bearing}
		\end{align}
		where  $p_i^{\mathcal{C}} = R_{c}\T (p_i^{\mathcal{B}}-p_{c})$ denotes the   coordinates  of the  $i$-th landmark expressed in frame $\{\mathcal{C}\}$.
	\end{itemize}  
	\begin{figure}
		\centering
		\subfloat[]{\includegraphics[width=0.44\linewidth]{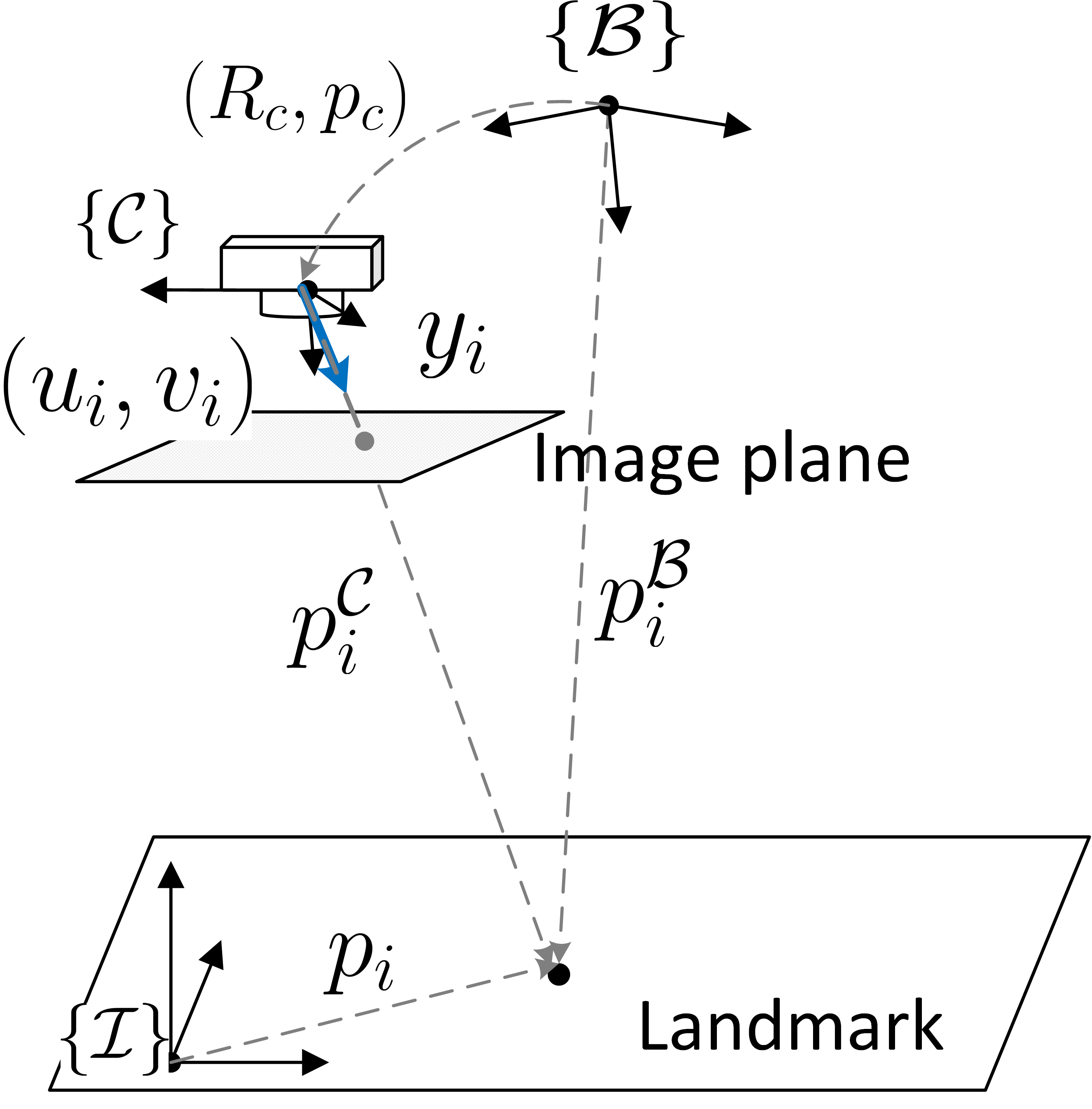}}~ 
		\subfloat[]{\includegraphics[width=0.48\linewidth]{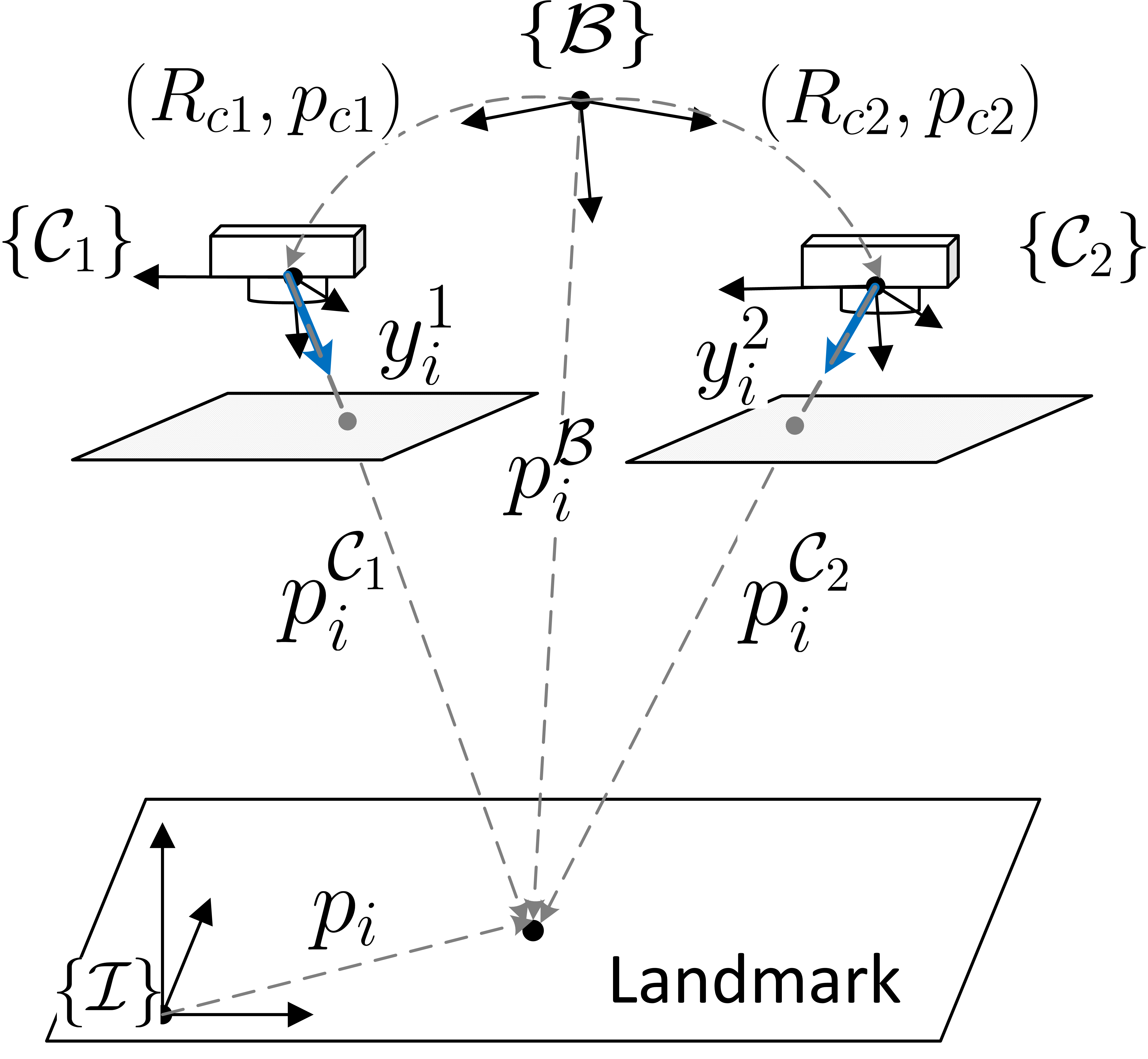}}\ 
		\caption{The geometry models of vision systems: (a) monocular vision system (b) stereo vision system.}
		\label{fig:visionsystems}
	\end{figure}
	\begin{remark}
		The bearing measurements of the $i$-th landmark can be obtained from the pixel measurements $(u_i,v_i)$  in the image plane as $y_i =  {\mathcal{K}^{-1} z_i}/{ \| \mathcal{K}^{-1} z_i \|}  \in \mathbb{S}^2$ with $z_i=[u_i,v_i,1]\T$ and $\mathcal{K}$ denoting the intrinsic matrix of the camera \cite{hartley2003multiple}. As we can see in  \eqref{eqn:stereo-bearing} and \eqref{eqn:monocular-bearing}, only partial information of the body-frame landmark positions are available in the monocular-bearing and stereo-bearing measurements. 
	\end{remark}
	
	\subsection{Objectives}
	Our main objectives in this work are as follows:
	\begin{itemize}
		\item[1)] Design an almost globally asymptotically stable nonlinear observer for the simultaneous estimation of the attitude $R(t)$, position $p(t)$ and linear velocity $v(t)$, using the IMU measurements ($\omega(t), a(t)$) and the visual measurements from either  \eqref{eqn:stereo-bearing} or \eqref{eqn:monocular-bearing}. The observer should be generic in the sense that it does not require any modification when using any of the above mentioned measurements. 
		
		\item[2)] Carry out a detailed uniform observability analysis for the above mentioned visual measurements scenarios  and provide sufficient feasibility conditions depending on the number and location of the landmarks, as well as the motion of the vehicle.  
		
		\item[3)] Provide a version of the observer with continuous IMU measurements, and sampled and intermittent visual measurements for practical implementation purposes. This is motivated by the fact that vision systems in general provide measurements at much lower sampling rates compared to the IMU sampling rates.
	\end{itemize}	
	
	\section{A Nonlinear Observer Using Continuous Vision-Based Measurements}\label{sec:IV} 
	\subsection{Nonlinear Observer Design}
	\begin{figure}[!ht]
		\centering
		\includegraphics[width=0.9\linewidth]{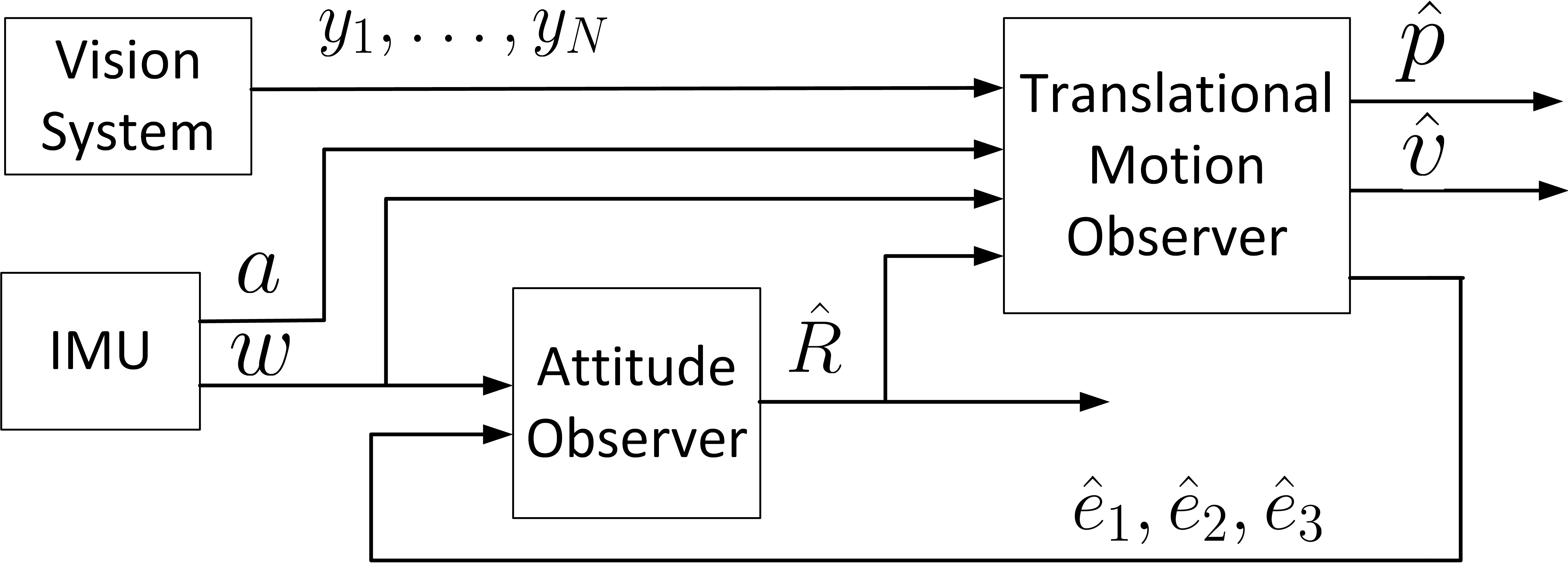}
		\caption{Structure of the proposed nonlinear observer on $SO(3) \times \mathbb{R}^{15}$ for vision-aided INSs.}
		\label{fig:diagram}
	\end{figure}
	
	To solve the vision-aided state estimation problem described in Section \ref{subsection:models}, we propose the following nonlinear  observer  on $SO(3) \times \mathbb{R}^{15}$: 
	
	\begin{subequations} \label{eqn:observer_1}	
		\begin{align} 
			\dot{\hat{R}}~  & = \hat{R}(\omega  +   \hat{R}\T \sigma_R )^\times \label{eqn:observer_1R} \\ 
			\dot{\hat{p}}~ & = \hat{v}  + \sigma_R^\times  \hat{p}   + \hat{R}   K_p \sigma_y \label{eqn:observer_1p} \\
			\dot{\hat{v}}~ & =  \hat{g} + \hat{R} a +  \sigma_R^\times \hat{v}  + \hat{R} K_v \sigma_y  \label{eqn:observer_1v}  \\ 
			\dot{\hat{e}}_i & =  \sigma_R^\times  \hat{e}_{i}  + \hat{R}K_i \sigma_y z\quad  i=1,2,3
			\label{eqn:observer_1e}  
		\end{align}
	\end{subequations}
	where $\hat{g}:=\sum\nolimits_{i=1}^{3} g_i \hat{e}_i$ with $g = [g_1,g_2,g_3]\T$. The rotation matrix $\hat{R}\in SO(3)$ denotes the estimate of the attitude $R$, and the vectors $\hat{p}\in \mathbb{R}^3$ and $\hat{v}\in \mathbb{R}^3$ denote the estimates of the position $p$ and linear velocity $v$, respectively. The structure of the proposed observer \eqref{eqn:observer_1} is shown  in Fig. \ref{fig:diagram}.  	
	The attitude innovation term $\sigma_R$ is given as follows:
	\begin{equation}      
		\sigma_R  :=   \frac{k_R}{2}  \hat{R}\sum_{i=1}^{3}   \rho_{i}    (\hat{R}\T \hat{e}_i)^\times (\hat{R}\T e_i) 
		=   \frac{k_R}{2}  \sum_{i=1}^{3}   \rho_{i}    \hat{e}_i^\times e_i 
		\label{eqn:innovation_R1} 
	\end{equation} 
	with constant scalars $k_R, \rho_i > 0, i=1,2,3$. 	
	The gain matrices $K_p,K_v,K_i\in \mathbb{R}^{3\times 3},i=1,2,3$ are designed as follows:
	\begin{equation}
		K  =  P C\T(t) Q(t) \label{eqn:Kdesign}
	\end{equation}
	with  $K:=[K_p\T, K_1\T, K_2\T,   K_3\T, K_v\T]\T $. The matrix $C(t)$ will be defined later depending on type of visual measurements used. The matrix $P$ is the solution to the following CRE:
	\begin{equation}
		\dot{P} = A(t)P + PA\T(t) - PC\T(t) Q(t) C(t)P + V(t) \label{eqn:CRE} 
	\end{equation}
	where $P(0) \in \mathbb{R}^{15\times 15}$ is a symmetric positive definite matrix, matrices $V(t)\in \mathbb{R}^{15\times 15}$ and $Q(t)\in \mathbb{R}^{3N\times 3N}$ are continuous, bounded and uniformly positive definite, and matrix $A(t)\in \mathbb{R}^{15 \times 15}$ is given as
	\begin{align}  
		A(t)  = 
		\begin{bmatrix}
			-\omega^\times  &  0_{3} & 0_{3} & 0_{3} &I_3   \\
			0_{3} & -\omega^\times  & 0_{3} & 0_{3} & 0_{3}    \\
			0_{3} & 0_{3} & -\omega^\times & 0_{3} & 0_{3}    \\
			0_{3} & 0_{3} & 0_{3} & -\omega^\times & 0_{3}  \\
			0_{3}  & g_{1} I_3 & g_2 I_3 & g_3 I_3 & -\omega^\times   
		\end{bmatrix}.  \label{eqn:closed-loop-A} 
	\end{align}
	The matrix $A(t)$ is obtained from the closed-loop translational error dynamics as it will be shown in the following subsection. 
	In the traditional Kalman filter, matrices $V(t)$ and $Q^{-1}(t)$ are associated to the covariance matrices of the additive noise on the system state and output, respectively. 
	The explicit design of $\sigma_y=[\sigma_{y 1}\T,\sigma_{y2}\T,\dots,\sigma_{yN}\T]\T\in \mathbb{R}^{3N}$ and the matrix $C(t)$ in the CRE \eqref{eqn:CRE}  for each type of visual measurement   described in Section \ref{subsection:models} are given as follows: 
	
	\begin{itemize}	
		\item [1)] \textit{Stereo-bearing measurements:}  From the stereo-bearing measurements defined in \eqref{eqn:stereo-bearing}, define the vector $\sigma_{yi}$ as 
		\begin{align} 
			\sigma_{yi} &=  \sum_{s=1}^2 \pi(R_{cs} y_{i}^s) ( \hat{R}\T  (\hat{p}_i-\hat{p} )-p_{cs}) \label{eqn:y_2}
		\end{align}
		for all $ i=1, 2, \dots, N$, with $p_i := [p_{i1}, p_{i2}, p_{i3}]\T=\sum_{j=1}^3 p_{ij} e_j$, $\hat{p}_i  = \sum_{j=1}^3 p_{ij} \hat{e}_j$ for all $i=1,2,\dots,N$ and  the projection map $\pi$ defined in \eqref{eqn:pi_x}. The matrix $C(t)$ is given by:
		\begin{align} \setlength\arraycolsep{1.5pt} 
			C(t)   =  
			\begin{bmatrix}
				\Pi_1  & -p_{11}\Pi_1 & -p_{12}\Pi_1 & -p_{13}\Pi_1 & 0_3    \\
				\Pi_2  & -p_{21}\Pi_2         & -p_{22}\Pi_2  & -p_{23}\Pi_2    & 0_3    \\
				\vdots  & \vdots& \vdots  & \vdots & \vdots   \\
				\Pi_N  & -p_{N1}\Pi_N         & -p_{N2}\Pi_N  & -p_{N3}\Pi_N & 0_3    
			\end{bmatrix} \label{eqn:C2}  
		\end{align}
		with $\Pi_i := \sum_{s=1}^2 \pi(R_{cs} y_{i}^s)\in \mathbb{R}^{3\times 3}, i\in\{1,\dots,N\}$. 
		\item [2)] \textit{Monocular-bearing measurements:} From the monocular-bearing measurements in \eqref{eqn:monocular-bearing}, vector $\sigma_{yi}$ is designed as
		\begin{equation}
			\sigma_{yi} =   \pi(R_{c} y_i) (\hat{R}\T  (\hat{ {p}}_i-\hat{p} )-p_{c})    \label{eqn:y_3}
		\end{equation}
		for all $ i=1, 2, \dots, N$, with the map $\pi$ defined in \eqref{eqn:pi_x}. The matrix $C(t)$ is given by:
		\begin{align}  \setlength\arraycolsep{1.5pt} 
			C(t)    =  
			\begin{bmatrix}
				\Pi_1  & -p_{11}\Pi_1 & -p_{12}\Pi_1 & -p_{13}\Pi_1 & 0_3    \\
				\Pi_2  & -p_{21}\Pi_2         & -p_{22}\Pi_2  & -p_{23}\Pi_2    & 0_3    \\
				\vdots  & \vdots& \vdots  & \vdots & \vdots     \\
				\Pi_N  & -p_{N1}\Pi_N         & -p_{N2}\Pi_N  & -p_{N3}\Pi_N & 0_3    
			\end{bmatrix} \label{eqn:C3}  
		\end{align}
		with $\Pi_i :=     \pi(R_c y_i)\in \mathbb{R}^{3\times 3}, i\in\{1,\dots,N\}$. 
	\end{itemize}
	
	\begin{remark}
		Note that the non-standard innovation term $\sigma_R$ in \eqref{eqn:innovation_R1}  relies on the inertial frame axes $e_i, i\in\{1,2,3\}$ and the auxiliary dynamical signals $\hat{e}_i, i\in\{1,2,3\}$. The motivation behind this construction is as follows. Typically, the attitude can be estimated using body-frame measurements of at least two non-collinear inertial frame vectors \cite{mahony2008nonlinear}. These body-frame vector measurements can be easily constructed from full landmark position measurements, see for instance \cite{wang2020nonlinear}. However, in the case of body-frame bearing measurements, the problem is quite challenging since we do not have the corresponding inertial vectors that will allow the construction of an appropriate innovation term $\sigma_R$. To overcome this challenge, we consider the inertial basis vectors $e_i$ and their corresponding body-frame vectors $R\T e_i$. Since $R\T e_i$ is unknown, we design the adaptive auxiliary vectors $\hat{e}_i$ such that $\hat{R}\T \hat{e}_i$ tends exponentially to $R\T e_i$. 
		This idea is somewhat similar to the idea of the velocity-aided attitude estimation schemes where we introduce an auxiliary variable to overcome the lack of the acceleration in the inertial frame \cite{hua2010attitude,roberts2011,berkane2017attitude}. 
		
	\end{remark}
	
	\begin{remark}
		The projection operator $\pi(R_c y_i) $ projects vectors onto the plane orthogonal to the body-frame bearing $R_c y_i$. This operator allows to eliminate, from the projected vector, the component which is collinear to $R_c y_i$. For instance, in \eqref{eqn:y_3} the projection of $\hat{R}\T  (\hat{ {p}}_i-\hat{p} )-p_{c}$ onto the plane orthogonal to the body-frame unit vector $R_cy_i$ will boil down to the projection of $\hat{R}\T  (\hat{ {p}}_i-\hat{p} )-p_i^{\mathcal{B}}$, since the component  $p_i^{\mathcal{B}}-p_c$, which is parallel to $R_cy_i$, will be eliminated. This mechanism, will allow us to generate the needed terms to put the error injection vector $\sigma_y$ in the sought-after form $\sigma_y=C(t)\tilde{x}$ as it will be shown later. The projection idea, however, is not new; it has been used in different ways in many references, for instance,  \cite{baldwin2009nonlinear,batista2015navigation,hamel2017position,berkane2020nonlinear}. 
	\end{remark}

	\begin{remark}
		In fact, landmark positions can be algebraically reconstructed from stereo-bearing measurements. The main motivation for the direct use of the stereo-bearing measurements in our observer is related to the robustness of the resulting estimation algorithm. As shown in the experimental results (Section \ref{sec:experimental}), the observers relying on landmark position measurements may fail in situations where one of the cameras of the stereo-vision system loses sight of the landmarks for some period of time. However, our observer using direct stereo-bearing measurements handles this situation very well by switching to a monocular bearing configuration.   
	\end{remark}

	\subsection{Error Dynamics and Stability Analysis} \label{sec:stability}
	Define the geometric attitude estimation error $\tilde{R} := R\hat{R}\T \in SO(3)$, and the translational vector state estimation error $ \tilde{x} := [
	{\tilde{p}}\T,  
	{\tilde{e}}_{1}\T,
	{\tilde{e}}_{2}\T, 
	{\tilde{e}}_3\T,
	{\tilde{v}}\T ]\T \in \mathbb{R}^{15}
	$ with $\tilde{p}  =   R\T p - \hat{R}\T \hat{p}$, $\tilde{v}= R\T v - \hat{R}\T\hat{v}$  and $\tilde{e}_i =  R\T e_i - \hat{R}\T \hat{e}_i, \forall i\in\{1,2,3\}$. These geometric estimation errors are motivated from \cite{wang2019nonlinear,wang2020hybrid,wang2020nonlinear}, and are different from the standard linear errors used in classical EKF-based filters \cite{mourikis2007multi,li2013high}.  
	Thus, the innovation term $ \sigma_R $ in \eqref{eqn:innovation_R1} can be rewritten as
	\begin{align}
		\sigma_R &= k_R\psi_a(M\tilde{R}) +   \Gamma(t) \tilde{x} \label{eqn:LTV_sigma_R} 
	\end{align}
	with $\Gamma(t)  :=   \frac{k_R}{2} [
	0_{3}, \rho_1 e_1^\times \hat{R}(t),   \rho_2 e_2^\times \hat{R}(t)
	,  \rho_3   e_3^\times \hat{R}(t), 0_{3}] \in \mathbb{R}^{3\times 15}$, 
	$M := \sum_{i=1}^3 \rho_i e_i e_i\T=\mathrm{diag}(\rho_1,\rho_2,\rho_3)$ and $
	\psi_a(M\tilde{R}) =-\frac{1}{2} \sum_{i=1}^3 \rho_i   e_i ^\times \tilde{R}\T e_i 
	$. It is not difficult to show that $\Gamma(t)$ is continuous and bounded since $\|\Gamma(t)\|_F  \leq   \frac{\sqrt{2}}{2} k_R\sum_{i=1}^{3}   \rho_{i}   : = c_\Gamma$.   Moreover, for any distinct non-negative scalars $\rho_i,i=1,2,3$, the matrix $M$ is positive semi-definite with three distinct eigenvalues.  From \eqref{eqn:LTV_sigma_R}, one can notice that $\sigma_R$ has two terms: the first term $k_R\psi_a(M\tilde{R})$ is commonly used for the establishment of the stability proofs of the attitude estimation subsystem; see for instance \cite{mahony2008nonlinear,berkane2017hybrid}. The second term depending on the estimation error $\tilde{x}$ is an asymptotically vanishing term as it will be shown later.
	In view of \eqref{eqn:INS_system}, \eqref{eqn:observer_1} and \eqref{eqn:LTV_sigma_R}, one obtains the following closed-loop system:
	\begin{subequations}\label{eqn:closed-loop}
		\begin{align} 
			\dot{\tilde{R}} &= \tilde{R}( -k_R\psi_a(M\tilde{R}) -   \Gamma(t) \tilde{x})^\times \label{eqn:dynamic_R} \\
			\dot{\tilde{x}} &= A(t) \tilde{x} - K \sigma_y \label{eqn:dynamic_x} 
		\end{align}
	\end{subequations}
	where  $A(t)$ is defined in \eqref{eqn:closed-loop-A}, and $K$ is designed in \eqref{eqn:Kdesign} relying on the solution $P(t)$ to the CRE \eqref{eqn:CRE}.  Note that the overall closed-loop system \eqref{eqn:closed-loop} is nonlinear and it can be seen as a cascade interconnection of a linear time-varying  (LTV) system on $\mathbb{R}^{15}$  and a nonlinear system evolving on $SO(3)$.  Given continuous and bounded matrices $A(t),C(t),Q(t),V(t)$, with $Q(t),V(t)$ being uniformly positive definite and the pair $(A(t),C(t))$ being uniformly observable, it follows that the solution $P(t)$ to the CRE \eqref{eqn:CRE} is well defined on $\mathbb{R}_{\geq 0}$ and there exist positive constants $0< p_m\leq p_M < \infty$ such that $p_m I_{15} \leq P(t) \leq p_M I_{15}$ for all $ t\geq 0$ \cite{bucy1967global,bucy1972riccati}.
	
	\begin{theorem} \label{theo:theo1}
		Consider the nonlinear system \eqref{eqn:closed-loop} with $\sigma_y=C(t)\tilde{x}$ and $C(t)$ being continuous and bounded. Let Assumption \ref{assum:IMU} hold, and suppose that the pair $(A(t),C(t))$ is uniformly observable. Pick $k_R>0$ and three distinct scalars $\rho_i > 0, i=1,2,3$. Let $K$ be given in \eqref{eqn:Kdesign} with matrices $Q(t)$ and $V(t)$ in \eqref{eqn:CRE} being continuous, bounded and uniformly positive definite. Then, the following statements hold:
		\begin{itemize}
			\item [i)] All solutions of the closed-loop system \eqref{eqn:closed-loop} converge to the set of equilibria given by $(I_3,0_{15\times 1})\cup \Psi_M$ where 
			\begin{align}
				\Psi_M:= \{(\tilde{R}, \tilde{x})\in SO(3)\times \mathbb{R}^{15}|  \tilde{R}=\mathcal{R}_\alpha(\pi,v),   \nonumber \\
				v\in \mathcal{E}(M),  \tilde{x}=0_{15\times 1} \}.
			\end{align} 
			\item [ii)] The desired equilibrium $(I_3,0_{15\times 1})$ is locally exponentially stable.
			
			\item [iii)]  All the undesired equilibria in $\Psi_M$ are unstable, and the desired equilibrium   $(I_3,0_{15\times 1})$ is almost globally asymptotically stable.  
		\end{itemize}
	\end{theorem}
	\begin{proof}
		See  Appendix \ref{sec:proof_theo1}.
	\end{proof}
	Theorem \ref{theo:theo1} provides AGAS and local exponential stability results for the proposed nonlinear observer. Among the interesting features of our observer is the fact that $\tilde{x}$ is guaranteed to converge globally exponentially to zero independently from the dynamics of $\tilde{R}$, as long as $\sigma_y$ can be written as $\sigma_y=C(t)\tilde{x}$ and the pair $(A(t),C(t))$ is uniformly observable. Note that $A(t)$ in \eqref{eqn:closed-loop-A}   is continuous and bounded since $\omega(t)$ is   continuous and bounded. It is worth pointing out that AGAS for \eqref{eqn:closed-loop} is the strongest result one can aim at with a smooth vector field on $SO(3)\times \mathbb{R}^{15}$. This is due to the topological obstruction on Lie group $SO(3)$, which consists in the fact that no continuous time-invariant vector field on $SO(3)$ leads to a globally asymptotically stable equilibrium \cite{koditschek1988application}.
	
	\subsection{Observability Analysis}\label{sec:obsv}    	
	In this subsection, we derive sufficient conditions for the uniform observability of the pair $(A(t),C(t))$ for the previously mentioned two types of vision-based measurements. An important technical result that will be used to carry out our uniform observability proofs is given in the following lemma:
	\begin{lemma}\label{lemma:lemmaUOC}
		Consider a constant matrix $A\in \mathbb{R}^{n\times n}$ and a (possibly) time-varying matrix  $C(t)\in  \mathbb{R}^{m\times n}$ such that: 
		\begin{itemize}
			\item[1)] All eigenvalues of $A$ are real.
			\item[2)] $C(t)$ is continuous and bounded. 
		\end{itemize}
		Let $N = A-S$ be a nilpotent matrix with index $s\leq n$, where $S$ is a diagonalizable matrix. Let $\mathcal{O}(t)\in \mathbb{R}^{r\times n}$ be a matrix composed of ($r>0$) row vectors of $C(t)$, $C(t)N, \dots, C(t)N^{s-1}$. Suppose that there exist  constant scalars $\delta,\mu>0$ such that 
		\begin{equation}
			\int_{t}^{t+\delta} \mathcal{O}\T(\tau) \mathcal{O}(\tau) d\tau > \mu I_n, \quad \forall t\geq 0.  \label{eqn:gramiancondition1}
		\end{equation}  
		Then, the pair $(A,C(t))$ is uniformly observable.	 
	\end{lemma} 
	\begin{proof}
		See  Appendix \ref{sec:proof_lemmaUOC}.
	\end{proof}		
	Note that the decomposition $A = S +N$ with  a nilpotent matrix $N$ and a diagonalizable matrix $S$ is known as the Jordan-Chevalley decomposition. It is important to mention that matrices $N$ and $S$  are uniquely determined and commute (\ie, $SN=NS$), see \cite[Theorem 1]{perko2013differential}. The main advantage of Lemma \ref{lemma:lemmaUOC} is to provide a relaxed condition for the uniform observability of the pair $(A,C(t))$ when the nilpotent part of $A$ is non-zero (\ie, $N\neq 0_n$). Condition   \eqref{eqn:gramiancondition1} is equivalent to the Kalman observability if $A$ is a nilpotent matrix and $C$ is constant. Moreover, if $A$ is a diagonalizable matrix (\ie, $N= 0_n$), condition \eqref{eqn:gramiancondition1} reduces to the persistency of excitation (PE) requirement on $C(t)$. Note that other uniform observability conditions were proposed in the literature such as \cite[Lemma 3.1]{scandaroli2013visuo}, which involves high-order derivatives of $C(t)$, and \cite[Lemma 2.7]{hamel2017position}, which is more suitable when $C(t)$ is the product of a PE matrix and a constant matrix guaranteeing Kalman observability.

	\subsubsection{Stereo-bearing measurements} \label{sec:Ob_stereo}  
	From \eqref{eqn:stereo-bearing},  one can rewrite $\sigma_{yi}$ in \eqref{eqn:y_2} in terms of the estimation errors as
	\begin{align}    
		\sigma_{yi} & =   \sum_{s=1}^2 R_{cs}\pi(y_{i}^s)R_{cs}\T     ( \hat{R}\T  (\hat{p}_i-\hat{p} )- R\T(p_i-p)    ) \nonumber \\
		&      =     \Pi_i \tilde{p}- p_{i 1}  \Pi_i \tilde{e}_1 -p_{i 2}  \Pi_i \tilde{e}_2 -p_{i 3}  \Pi_i \tilde{e}_3  \label{eqn:def_y_i_1-2}
	\end{align}
	for all $i=1,\dots,N$, where we made use of  the facts $\Pi_i  = \sum_{s=1}^2 \pi(R_{cs} y_{i}^s)  $,   $p_i = \sum_{j=1}^3 p_{i j} e_j$ and  $ \pi(R_{cs} y_{i}^s)  (R\T(p_i-p) - p_{cs})=0_{3\times 1}$. From the definition of $\tilde{x}$, one obtains  $\sigma_y=C(t)\tilde{x}$  with $C(t)$ defined in \eqref{eqn:C2}.	 For each $i\in\{1,\dots,N\}$, the matrix $\Pi_i$ is positive definite if the vectors $R_{c1} y_{i}^1$ and $R_{c2} y_{i}^2$ are non-collinear. Note that for stereo vision systems, $R_{c1} y_{i}^1$ and $R_{c2} y_{i}^2$ are naturally non-collinear since all the visible landmarks are within the limited sensing distance in practice. Hence, one obtains that the  matrices $\Pi_i, i=1,\dots,N$ are uniformly positive definite.  
	Moreover, the matrix $C(t)$ is continuous since $\Pi_i(t), \forall i=\{1,\dots,N\}$ is continuous. Defining the  block diagonal matrix 
	$\Theta(t):=\blkdiag(\Pi_1(t),\dots,\Pi_N(t))$, one has $C(t) = \Theta(t) \bar{C} $ with  
	\begin{align} 
		\bar{C} &  =  
		\begin{bmatrix}
			I_3  & -p_{11}I_3  & -p_{12}I_3 & -p_{13}I_3  & 0_{3}   \\
			I_3  & -p_{21}I_3         & -p_{22}I_3  & -p_{23}I_3    & 0_{3}     \\
			\vdots  & \vdots& \vdots  & \vdots    \\
			I_3  & -p_{N1}I_3         & -p_{N2}I_3  & -p_{N3}I_3 & 0_{3}  
		\end{bmatrix} \label{eqn:Cbar} .
	\end{align}
	From {\eqref{eqn:Cbar}}, it is easy to show that the matrix $C(t)$ defined in  {\eqref{eqn:C2}} is bounded since $\|{\Theta}(t)\|_F \leq   \sum_{i=1}^N \|\Pi_i(t)\|_F$ and $\|\Pi_i(t)\|_F,i=1,\dots,N$ are bounded for all $t\geq 0$.
	
	\begin{lemma} \label{lemma:AC2}
		Consider the matrices $A(t)$ defined in \eqref{eqn:closed-loop-A} and $C(t)$ defined in \eqref{eqn:C2}. Suppose that  there exist three non-aligned landmarks among the $N\geq 3$ measurable landmarks, whose plane is not parallel to the gravity vector.
		Then, the pair $(A(t),C(t))$ is uniformly observable.  
	\end{lemma}
	\begin{proof}
		See  Appendix \ref{sec:lemmaAC2}.
	\end{proof}
	\begin{remark}
		The sufficient observability conditions in Lemma \ref{lemma:AC2} using stereo bearing measurements are mildly stronger than those needed for the local observers in the literature \cite{barrau2017invariant,hua2018riccati,wang2020nonlinear} using 3D landmark position measurements. This is mainly due to the over-parameterization of our observer with the additional auxiliary signals $\hat{e}_i$, which is the paid price for the almost global asymptotic stability results achieved in Theorem \ref{theo:theo1}.
	\end{remark}

	\subsubsection{Monocular-bearing measurements}  \label{sec:Ob_monocular}
	From \eqref{eqn:monocular-bearing},  one can rewrite $\sigma_{yi}$ in \eqref{eqn:y_3} in terms of the estimation errors as
	\begin{align} 
		\sigma_{yi}  	& =   R_{c}\pi(y_i)R_{c}\T     ( \hat{R}\T  (\hat{p}_i-\hat{p} )-p_{c} -  (R\T(p_i-p) - p_{c})   ) \nonumber \\
		&  =          \Pi_i \tilde{p}- p_{i 1}  \Pi_i \tilde{e}_1 -p_{i 2}  \Pi_i \tilde{e}_2 -p_{i 3}  \Pi_i \tilde{e}_3   \label{eqn:def_y_i_2-2}
	\end{align}
	for all $i=1,\dots,N$   with $\Pi_i  =    \pi(R_c  y_i)  $. From the definition  of $\tilde{x}$,  one  obtains $\sigma_y = C(t) \tilde{x}$ with $C(t)$ defined in \eqref{eqn:C3}.
	Note that the matrix $C(t)$  in \eqref{eqn:C3}  is similar to the one in \eqref{eqn:C2}, and the main difference is that the   matrix $\Pi_i$ in \eqref{eqn:C3} is only guaranteed to be positive semi-definite. One can also show that $C(t) = \Theta(t) \bar{C} $ with $\bar{C}$  defined in \eqref{eqn:Cbar} and $\Theta(t)=\blkdiag(\Pi_1,\Pi_2,\dots,\Pi_N)$ being continuous and bounded.  
	\begin{lemma} \label{lemma:AC3}
		Consider the matrices $A(t)$ defined in \eqref{eqn:closed-loop-A} and $C(t)$ defined in \eqref{eqn:C3}. Suppose that there exist three non-aligned landmarks, indexed by $\ell_1,\ell_2,\ell_3$,  among the $N\geq 3$ measurable landmarks, whose plane is not parallel to the gravity vector, and one of following statements holds:
		\begin{itemize}
			\item [i)]  The camera is in motion with bounded velocity and there exists a constant $\epsilon>0$ such that for any time $t^*\geq 0$ and landmark $i\in\{\ell_1,\ell_2,\ell_3\}$, there exists some time $t>t^*$ such that   $\|(R(t)R_c y_i(t)) \times (R(t^*)R_c y_i(t^*))\| \geq \epsilon $. 
			\item [ii)] The camera is motionless and the following matrix has full rank of $15+N$
			\begin{align} \mathcal{O}' = 
				\begin{bmatrix}
					\bar{C} & M\\
					N_1& 0_{3\times N}\\
					N_2 & 0_{3\times N}
				\end{bmatrix} \in  \mathbb{R}^{(3N+6)\times (15+N)}   \label{eqn:Motionless}
			\end{align}
			where $\bar{C}$ is defined in \eqref{eqn:Cbar},  $N_1:=[0_3, 0_3, 0_3, 0_3,I_3, 0_3]$, $ N_2 :=[0_3, g_1 I_3, g_2 I_3, g_3 I_3, 0_3]$, and $M := \blkdiag(p_1-p',\dots,p_N-p')$ with $p'=p+Rp_c$  denoting  the position of the camera in the inertial frame.
		\end{itemize}
		Then, the pair $(A(t),C(t))$ is uniformly observable. 		 
	\end{lemma}
	\begin{proof}
		See  Appendix \ref{sec:lemmaAC3}.
	\end{proof} 
	\begin{remark}
		The proof of this lemma relies on the application of the technical Lemma \ref{lemma:lemmaUOC}. For the condition i), using the fact $R(t)R_c y_i(t) = {(p_i-p'(t))}/{\|p_i-p'(t)\|}$, it follows that the camera is not indefinitely moving in a straight line passing through one of landmarks $\ell_1,\ell_2,\ell_3$. Note that this condition involves an extra condition on the motion of the camera, with respect to Lemma  \ref{lemma:AC2}, to generate sufficient information for uniform observability. 	 	
		Condition ii) also holds when the camera is in motion and the minimum number of required landmarks is $5$ for the  matrix $\mathcal{O}'$ to have a full rank of $N+15$.
	\end{remark}
	\begin{proposition} \label{pro:non-observable}
		Consider the case where the camera is motionless with  $N\geq 5$ measurable non-aligned landmarks. Suppose that none of the following scenarios hold:
		\begin{itemize} 
			\item[(a)] All the landmarks are located in the same plane.				
			\item[(b)] There exist three non-aligned landmarks and the rest of landmarks are located in the plane parallel to the gravity vector and contains two of these three landmarks.
			\item[(c)] There exist three non-aligned landmarks and the rest of the landmarks are aligned with one of these three landmarks and the position of the camera. 
			\item[(d)] There exist three non-aligned landmarks and the rest of the landmarks are either located in the plane parallel to the gravity vector and contains two of these three landmarks, or aligned with the third of these three landmarks and the position of the camera.
		\end{itemize}
		Then, the matrix  $\mathcal{O}'$  defined in \eqref{eqn:Motionless}  has full rank.
	\end{proposition}
	
	\begin{proof}
		See Appendix \ref{sec:non-observable}
	\end{proof}
	All the four cases stated in Proposition \ref{pro:non-observable} are summarized in Fig \ref{fig:non-observable}. Note that the cases (a) and (b) are independent from the location of the camera.  
	According to the Lemma \ref{lemma:AC3}, if the matrix  $\mathcal{O}'$ has full rank, one can conclude that the pair $(A(t),C)$ is uniformly observable. 
	Note that the state estimation in this static case is also known as the  static Perspective-n-Point (PnP) problem \cite{lepetit2009epnp,hamel2018riccati}. In this scenario, one aims at determining the pose (position and orientation) of a camera given its intrinsic parameters and a set of $N$ correspondences between 3D points and their 2D projections. 
	\begin{figure}[!ht]
		\centering
		\subfloat[]{\includegraphics[width=0.38\linewidth]{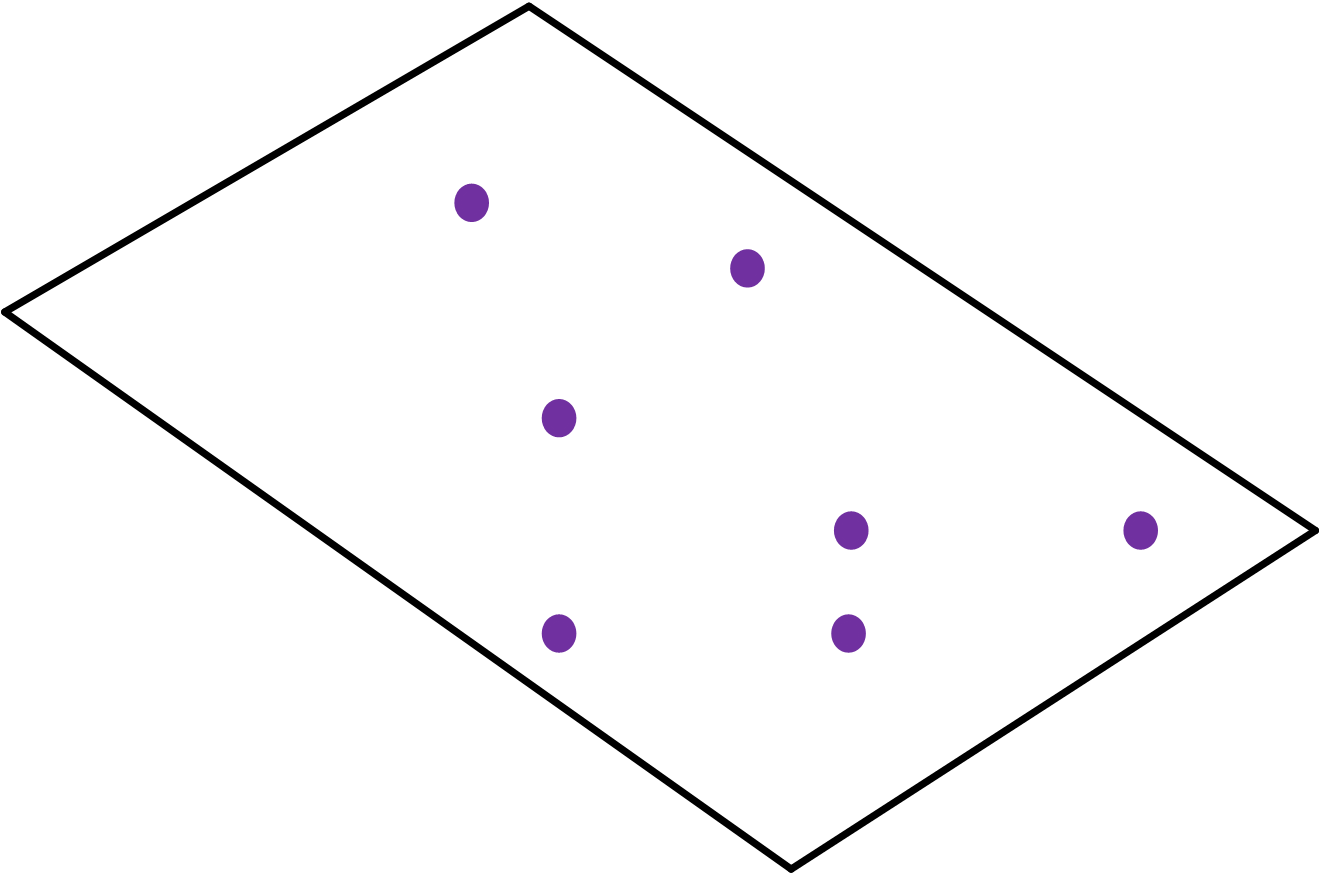}}~~~
		\subfloat[]{\includegraphics[width=0.41\linewidth]{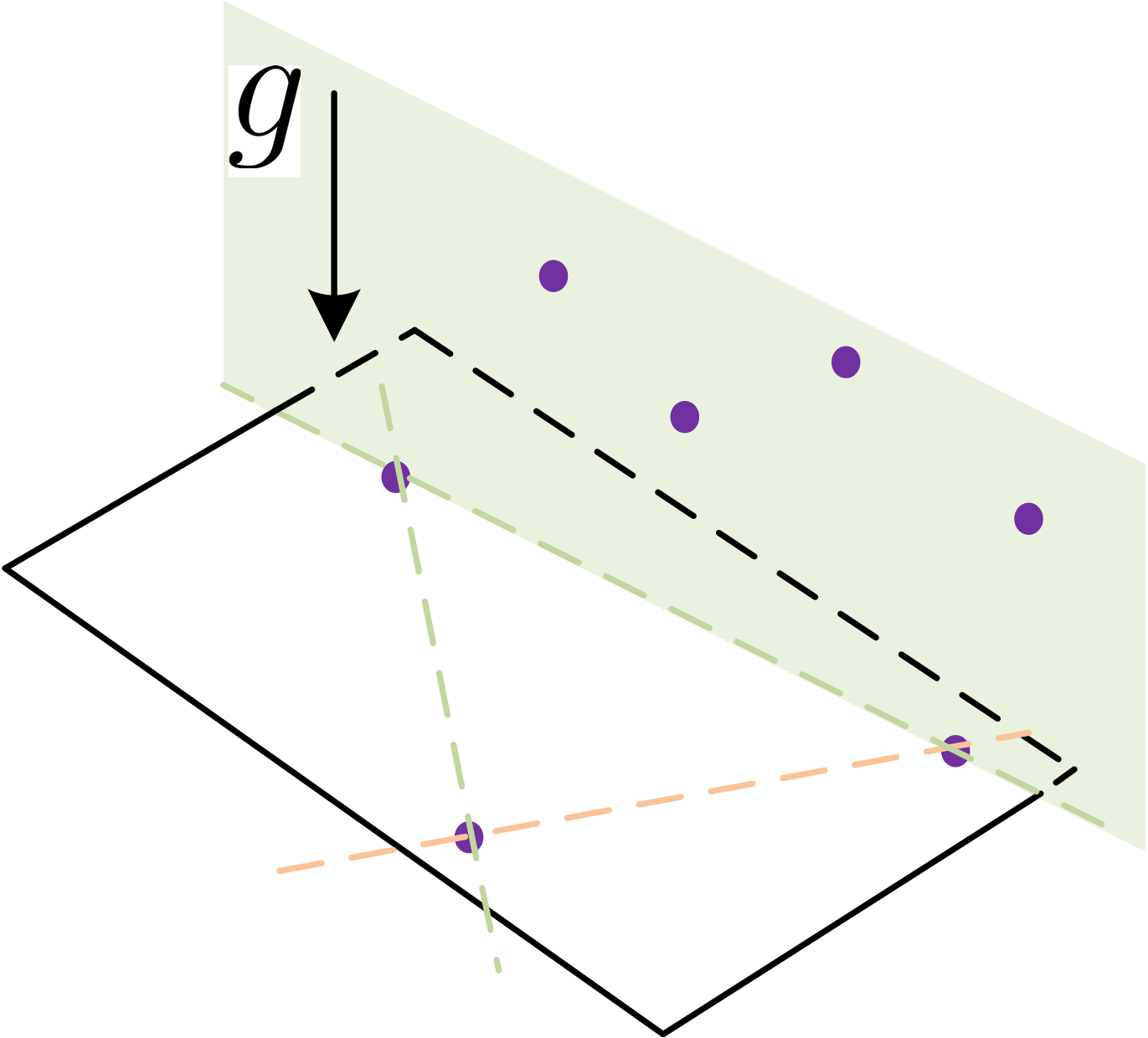}}\\[-0.35cm]
		\subfloat[]{\includegraphics[width=0.45\linewidth]{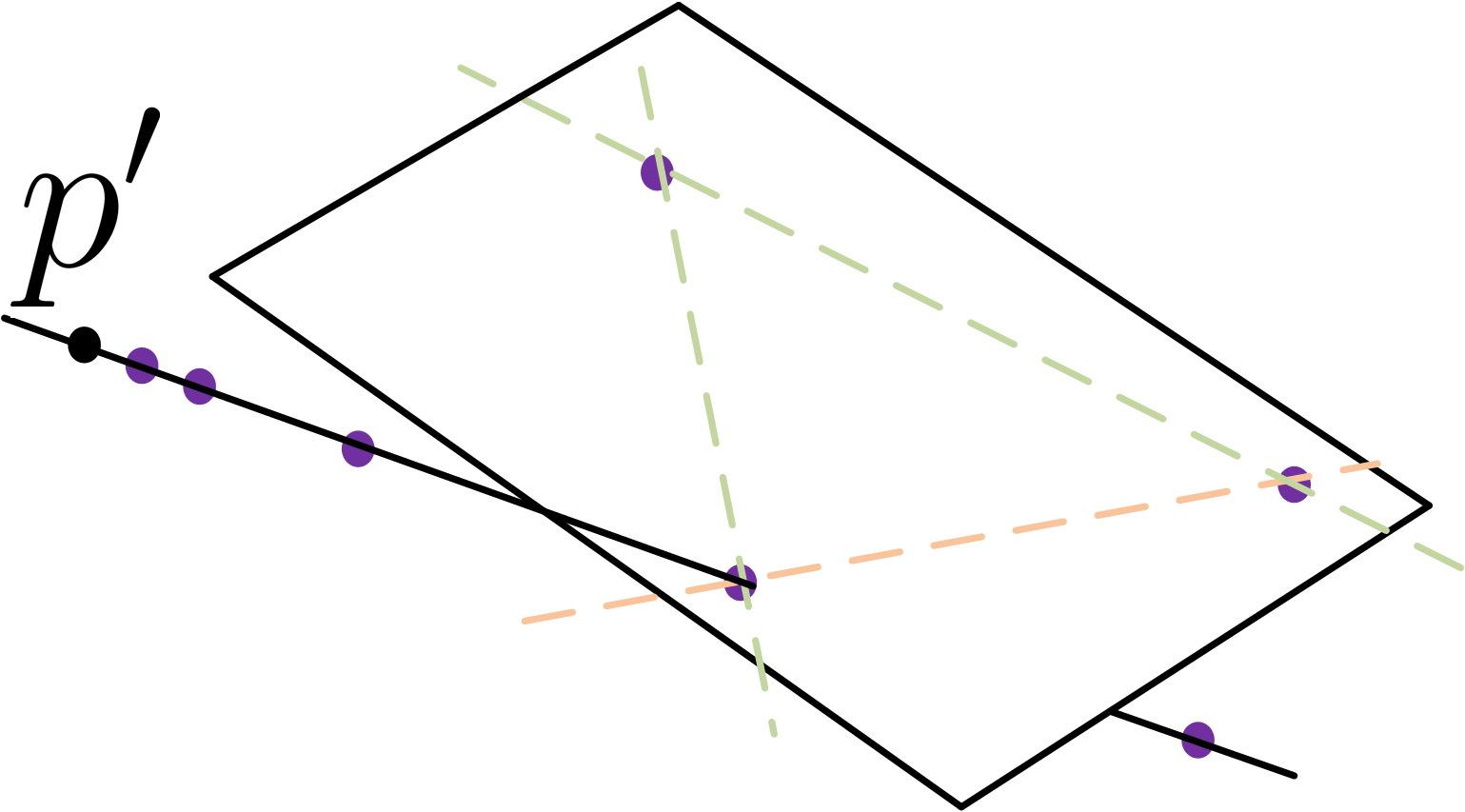}}~
		\subfloat[]{\includegraphics[width=0.46\linewidth]{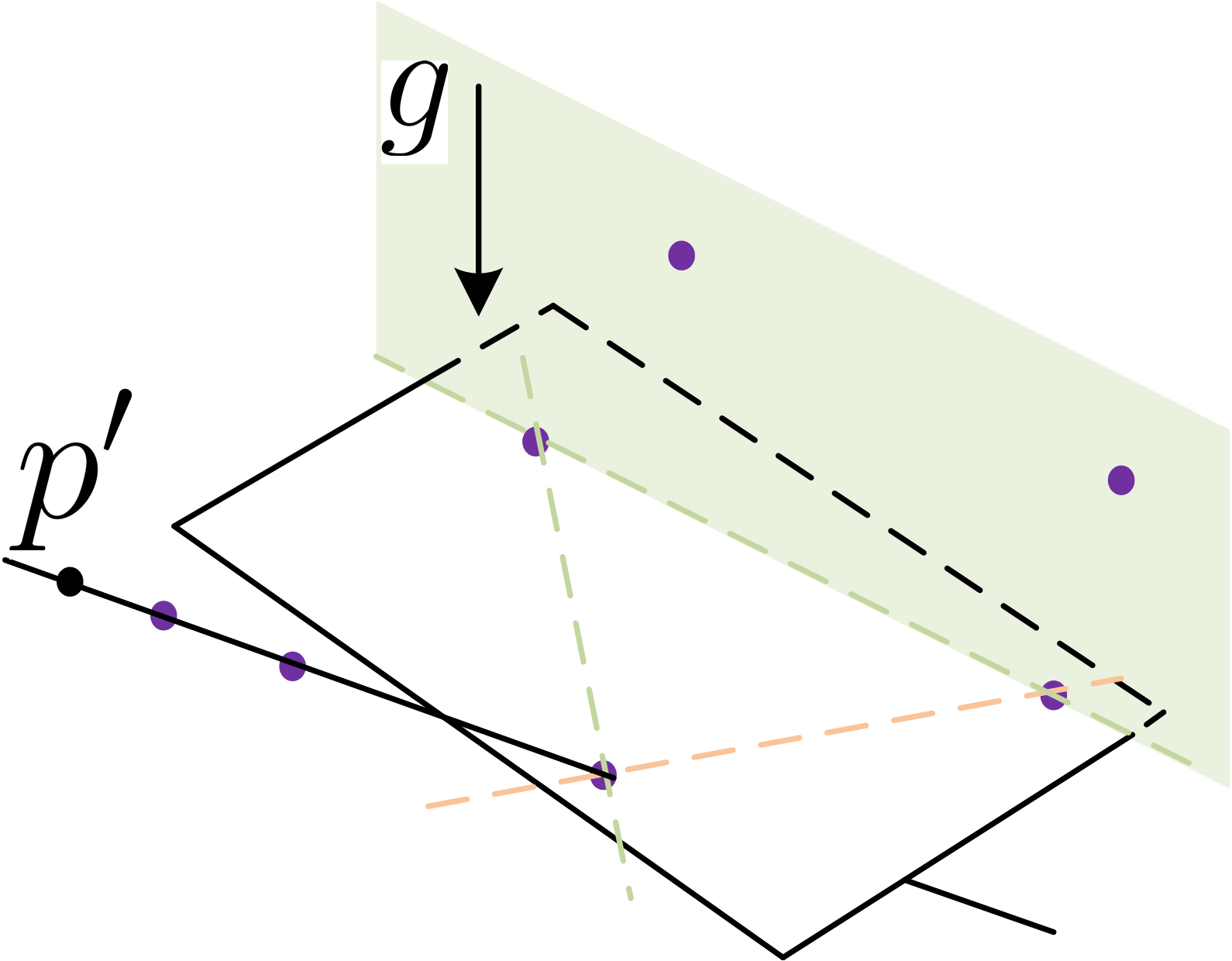}}
		\caption{The four possible cases of Proposition \ref{pro:non-observable} where the matrix  $\mathcal{O}'$  in \eqref{eqn:Motionless} does not have full rank. The locations of  $N\geq 5$ non-aligned landmarks are depicted by purple dots, and the location of the motionless monocular camera is depicted by a black dot.}
		\label{fig:non-observable}
	\end{figure}

	\section{Hybrid Observer   Using Intermittent Vision-Based Measurements} \label{sec:V}
	
	In practical applications, the IMU measurements can be obtained at a high rate, while the vision-based measurements are often obtained at a much lower rate due to the hardware design of the vision sensors and the heavy image processing computations. Hence, the IMU measurements can be assumed as continuous and the measurements from the vision systems are sampled intermittently. This motivates us to redesign the proposed continuous nonlinear observer in terms of continuous IMU and intermittent vision-based measurements. 
	
	\begin{assumption}
		We assume that the vision-based measurements are available at strictly increasing time instants $\{t_{k}\}_{k\in \mathbb{N}_{>0}}$, and there exist constants $0<T_m\leq T_M<\infty$ such that $t_1 < T_M$ and $T_m \leq t_{k+1} -t_{k} \leq T_M$ for all $k\in \mathbb{N}_{>0}$. 
	\end{assumption} 
	This assumption implies that the time difference between two consecutive vision-based measurements are lower and upper bounded. The positive lower bound $T_m$ is required to avoid Zeno behaviors. Note that, if $T_m = T_M$, the vision-based measurements are sampled periodically.
	
	Motivated by the work in \cite{wang2020nonlinear} and making use of the framework of hybrid dynamical systems presented in \cite{goebel2009hybrid,goebel2012hybrid}, we propose the following hybrid nonlinear observer modified from  \eqref{eqn:observer_1} as
	\begin{align}
		\underbrace{\begin{array}{ll}
				\dot{\hat{R}}~  & = \hat{R}(\omega +\hat{R}\T \sigma_R)^\times \\ 
				\dot{\hat{p}}~  & = \sigma_R^\times  \hat{p} + \hat{v}   \\
				\dot{\hat{v}}~  & = \sigma_R^\times \hat{v}  +   \hat{g}   + \hat{R} a   \\ 
				\dot{\hat{e}}_i & = \sigma_R^\times  \hat{e}_{i},   \quad i\in{1,2,3}    
		\end{array}}_{t\in[t_{k-1},t_k],~ k\in \mathbb{N}_{>0}}	
		~
		\underbrace{\begin{array}{ll}
				\hat{R}^+  & = \hat{R}  \\ 
				\hat{p}^+ & =  \hat{p}   + \hat{R}   K_p \sigma_y  \\
				\hat{v}^+ & =   \hat{v}  + \hat{R} K_v \sigma_y  \\ 
				\hat{e}_i^+ & =  \hat{e}_{i}  + \hat{R}K_i \sigma_y  
		\end{array}	}_{t\in\{t_k\},~ k\in \mathbb{N}_{>0}}	
		\label{eqn:observer_2}  
	\end{align}
	where $k_R>0$, $\hat{g}:=\sum\nolimits_{i=1}^{3} g_i \hat{e}_i$, $\sigma_R$ is given in \eqref{eqn:innovation_R1}, and the vector $\sigma_y$ is given in \eqref{eqn:y_2} for stereo-bearing measurements and in \eqref{eqn:y_3} for monocular-bearing measurements. Let $K:=[K_p\T, K_1\T, K_2\T,K_3\T, K_v\T]\T $, which is designed as
	\begin{equation}
		K = PC\T(t)(C(t)PC\T(t) + Q^{-1}(t))^{-1} \label{eqn:Kt2}
	\end{equation} 	 
	where $P$ is the solution to the following  continuous-discrete Riccati equations (CDRE) 
	\begin{subequations}\label{eqn:CDRE}
		\begin{align}
			\dot{P} ~~&= A(t)P + PA(t)\T + V(t),    &t\in [t_{k},t_{k+1}] \label{eqn:CDRE-C}\\
			P^+ &= (I-KC(t)) P, &t \in \{t_k\} \label{eqn:CDRE-D}
		\end{align}
	\end{subequations}
	where  $A(t)$ is given by  \eqref{eqn:closed-loop-A}, $P(0)$ is symmetric positive definite, and $Q(t)  \in \mathbb{R}^{3N\times 3N},  V(t) \in \mathbb{R}^{15\times 15}$ are continuous, bounded and  uniformly positive definite. Then, as per \cite{deyst1968conditions,barrau2017invariant} and \cite[Lemma 7]{wang2020nonlinear}, the solution $P$ to  \eqref{eqn:CDRE} exists, and there exist constants $0< p_m \leq p_M < \infty $ such that $p_m I_{15} \leq P \leq p_M I_{15}$ for all $ t \geq 0$ if there exist positive constants $\mu\in \mathbb{R}$ and $\delta\in \mathbb{N}_{>0}$ such that 
	$
	W_o^h(t_j,t_{j+\delta}) = \sum_{i=j}^{j+\delta} \Phi\T(t_i,t_j) C\T(t_i) C(t_i) \Phi(t_i,t_j) > \mu I_{15}, \forall j\in \mathbb{N}_{>0},
	$ 
	where $\Phi(t,\tau)$ denotes the state transition matrix associated to $A(t)$. Note that $W_o^h$ is the discrete version of $W_o$ in \eqref{eqn:gramiancondition}. Hence, the same conditions as in Lemma \ref{lemma:AC2} and \ref{lemma:AC3} can be derived for the existence of the positive constants $\mu\in \mathbb{R}$ and $\delta\in \mathbb{N}_{>0}$  such that $W_o^h(t_j,t_{j+\delta})   > \mu I_{15}$ for all $ j\in \mathbb{N}_{>0}$.  
	
	In view of \eqref{eqn:INS_system}, \eqref{eqn:innovation_R1} and \eqref{eqn:observer_2}, one obtains the following hybrid closed-loop system: 
	\begin{subequations}\label{eqn:hybridclosed-loop}	
		\begin{align}  
			&\left.
			\begin{array}{rl}
				\dot{\tilde{R}}  &= \tilde{R}( -k_R \psi_a(M\tilde{R})-\Gamma(t) \tilde x)^\times \\
				\dot{ \tilde{x}}  & =A(t)  \tilde{x}
			\end{array}\right\}  &   t\in [t_{k},t_{k+1}] 	\label{eqn:hybridclosed-loop-flow}\\
			&\left.
			\begin{array}{ll}
				\tilde{R}^+ &= \tilde{R} \\
				\tilde{x}^+  &= (I_{15}-KC(t))  \tilde{x}  
			\end{array}\right\}   &   t \in \{t_k\}  
			\label{eqn:hybridclosed-loop-jump}
		\end{align}
	\end{subequations}	
	where the AGAS proof for the equilibrium $(I_3,0_{15\times 1})$ of hybrid system \eqref{eqn:hybridclosed-loop} can be easily conducted by combing the proof of Theorem \ref{theo:theo1} and \cite[Theorem 9]{wang2020nonlinear}, and is therefore omitted here. The proposed hybrid nonlinear observer for vision-aided INSs   has been summarized in Algorithm \ref{algo:1}.

	\begin{algorithm}[!ht] 
		\caption{ Nonlinear observer for vision-aided INSs}
		\begin{algorithmic}[1] \label{algo:1}
			\renewcommand{\algorithmicrequire}{\textbf{Input:}}
			\renewcommand{\algorithmicensure}{\textbf{Output:}}
			\REQUIRE  Continuous IMU measurements, and intermittent visual measurements  at the time instants $\{t_k\}_{k\in \mathbb{N}_{>0}}$. 		
			\ENSURE  $\hat{R}(t),\hat{p}(t)$ and $\hat{v}(t)$ for all $t\geq 0$
			\FOR {$k\geq 1$}  	
			\WHILE{$t\in [t_{k-1},t_k]$} \setstretch{1.1}
			\STATE \(\dot{\hat{R}}  = \hat{R}(\omega +    \hat{R}\T \sigma_R)^\times  \)    
			\mycomment{ $\sigma_R$ defined in \eqref{eqn:innovation_R1} }
			\STATE \(\dot{\hat{p}} ~ = \sigma_R^\times \hat{p}  + \hat{v} \)  
			\STATE \(\dot{\hat{v}} ~ = \sigma_R^\times \hat{v} + \sum_{i=1}^3 g_i \hat{e} 
			_i  + \hat{R}a  \)  
			\STATE \(\dot{\hat{e}}_i  = \sigma_R^\times \hat{e}_i\)   
			\mycomment{ for all $i=1,2,3$ }
			\STATE \(\dot{P}  =  {A}(t) P+ P {A}\T(t) + V(t) \)  
			\mycomment{ $ {A}$ defined  in \eqref{eqn:closed-loop-A}  and $V(t)$ being uniformly positive definite \hfill }
			\ENDWHILE 	 	 
			\STATE Obtain the vector $\sigma_y$ and matrix $C(t)$ from the visual measurements at time $t_k$
			\mycomment{Using \eqref{eqn:y_2} and \eqref{eqn:C2} for stereo-bearing measurements, or \eqref{eqn:y_3} and \eqref{eqn:C3} for monocular-bearing measurements \hfill }  
			\STATE 	\(K = P {C}\T(t_k)({C}(t_k)P {C}\T(t_k) + Q^{-1}(t_k))^{-1} \)  
			\mycomment{ $Q(t)$ being uniformly positive definite \hfill}
			\STATE 	Compute matrices $K_p,K_1,K_2,K_3$ and $K_v$ from $K$   \mycomment{using $K=[
				K_p\T, K_1\T, K_2\T,   K_3\T, K_v\T]\T $ \hfill }  
			\STATE \( \hat{R}^+~ =\hat{R}\)  
			\STATE \(\hat{p}^+~ = \hat{p} +   \hat{R}K_p  \sigma_y \) 
			\STATE \( \hat{v}^+~ = \hat{v} +    \hat{R}K_v    \sigma_y  \) 
			\STATE \(\hat{e}_{i}^+~  =   \hat{e}_{i}  + \hat{R}K_{i} \sigma_y \) 
			\mycomment{ for all $i=1,2,3$ }
			\STATE \( P^+~ = (I_{15}-KC)P\)  				 
			\ENDFOR
		\end{algorithmic}
	\end{algorithm} 
	\begin{remark}
		Note that our proposed observer \eqref{eqn:observer_2} is deterministic and the gain matrix $K$ is designed based on the CDRE \eqref{eqn:CDRE} with any uniformly positive definite matrices $V(t)$ and $Q(t)$. However, it is possible to relate (locally) the matrices $V(t)$ and $Q(t)$ to the measurements noise properties leading to a local sub-optimal design of the observer gains in the spirit of the Kalman filter. As in \cite{wang2020nonlinear}, let $\omega$ and $a$ be the noisy measurements and replace $\omega$ by $\omega+n_\omega$ and $a$ by $a+n_a$ in \eqref{eqn:INS_system}   with $n_\omega$ and $n_a $ denoting the noise signals associated to $\omega$ and $a$, respectively.  In view of \eqref{eqn:INS_system}, \eqref{eqn:observer_2} and \eqref{eqn:hybridclosed-loop-jump},   the time derivative of $\tilde{x}$  can be approximated around $\tilde{x}=0$ by $\dot{\tilde{x}} \approx A(t) \tilde{x} + G_t n_x$ 	with  $n_x = [n_\omega\T,n_a\T]\T$, $A(t)$ defined in \eqref{eqn:closed-loop-A} and
		\begin{align*} \arraycolsep=1.5pt\def\arraystretch{1.1}
			G_t= 
			-\begin{bmatrix}
				(\hat{R}\T \hat{p})^\times & 
				(\hat{R}\T \hat{e}_1)^\times &  
				(\hat{R}\T \hat{e}_2)^\times &  
				(\hat{R}\T \hat{e}_3)^\times & 					
				(\hat{R}\T \hat{v})^\times  \\
				0_{3} & 0_{3} & 0_{3} & 0_{3} & -I_{3}
			\end{bmatrix}\T . 
		\end{align*} 
		Moreover, replacing $y_i$ in \eqref{eqn:monocular-bearing} by $y_i + n_{y_i}$ for each   bearing measurement with $n_{y_i}$  denoting the noise signals. From \eqref{eqn:y_3}, $\sigma_{y_i}$ can be rewritten in the form of $\sigma_{y_i} = \Pi_i \tilde{p} - \sum_{j=1}^3\Pi_i \tilde{e}_j + \|p-p_i\|\Pi_i n_{y_i}$, and $\sigma_y$ can be approximated around $\tilde{x}=0$ by $\sigma_y   \approx C(t)\tilde{x} + M_t  n_y$ with $n_y = [n_{y_1}\T, \dots, n_{y_N}\T]\T$ and $M_t = \blkdiag(\|\hat{p}-\hat{p}_1\|  \Pi_1, \dots,   \|\hat{p}-\hat{p}_N\|  \Pi_N)$. Then, matrices $V(t)$ and $Q(t)$ can be related to the covariance matrices of the measurements noise as follows:
		\begin{subequations}\label{eqn:V-Q}
			\begin{align}
				V(t)  &= G_t \Cov\left( n_x\right) G_t\T \label{eqn:V-Q_V}\\ 
				Q^{-1}(t)  &=  M_t\Cov(n_y)M_t\T. \label{eqn:V-Q_Q}
			\end{align}	
		\end{subequations}
		In practice, a small positive definite matrix can be added to $V(t)$ and $Q^{-1}(t)$ to guarantee that $V(t)$ and $Q^{-1}(t)$ are uniformly positive definite.
	\end{remark}

	\section{Simulation results}\label{sec:simulation} 
	
	In this simulation, we consider an autonomous vehicle moving on the `8'-shape trajectory described by $p(t) = 2[\sin(t), \sin(t)\cos(t), 1]\T$ (m)  with the initial attitude $R(0)=I_3$ and angular velocity  $\omega(t) =  [-\cos(2t), 1, \sin(2t)]\T$ (rad/s). Let $g=[0,0,-9.81]\T$  (m/s$^2$) be the gravity vector expressed in the inertial frame, and assume that there are $N=5$ landmarks randomly selected. Different types of continuous vision-based measurements are generated using \eqref{eqn:stereo-bearing} and \eqref{eqn:monocular-bearing}. 
	For comparison purposes, 3D landmark position measurements $y_i=   R\T(p_i-p), i= 1,2 ,\dots,N$ are considered.
	In this case, the vector $\sigma_y=[\sigma_{y 1}\T,\sigma_{y2}\T,\dots,\sigma_{yN}\T]\T\in \mathbb{R}^{3N}$ used in \eqref{eqn:observer_1} is designed as
	\begin{equation}
		\sigma_{yi}  =  \hat{R}\T(\hat{p}_i - \hat{p}) - y_i, \quad    i=1,2,\dots,N  \label{eqn:y_1}
	\end{equation}
	with $\hat{p}_i  = \sum_{j=1}^3 p_{ij} \hat{e}_j$. 
	It follows that $\sigma_y = C \tilde{x}$ with a constant matrix $C=\bar{C}$ and $\bar{C}$ defined in \eqref{eqn:Cbar}\footnote{In practice, 3D landmark positions can be obtained, for instance, using  stereo cameras, and the sufficient condition for uniform observability of the pair $(A(t),C)$ with $A(t)$   in \eqref{eqn:closed-loop-A} is the same as the one in Lemma \ref{lemma:AC2}.}. 
	
	The same initial conditions are considered for each case as: $\hat{R}(0) = \exp(0.5\pi u^\times)$ with $u\in \mathbb{S}^2$, $\hat{v}(0)=\hat{p}(0) =  0_{3\times 1}$, $\hat{e}_i(0) = e_i, i=1,2,3$, and $P(0) = I_{15}$. The gain parameters for each case are chosen as $\rho_1 = 0.5, \rho_2 = 0.3, \rho_3=0.2$ and $k_R = 1$ for $\sigma_R$ in \eqref{eqn:innovation_R1}, and $Q  = 10^3I_{3N}, V  =  10^{-4}I_{15} $ for the CRE \eqref{eqn:CRE}. Simulation results are shown in Fig. \ref{fig:simulation_error}. As one can see, the estimated states from all the cases converge, after a few seconds, to the vicinity of the real states. Roughly speaking, with the same tuning parameters, the performances of the proposed continuous observer with three types of vision-based measurements is quite similar.

	\begin{figure}[!ht]
		\centering
		\includegraphics[width=0.91\linewidth]{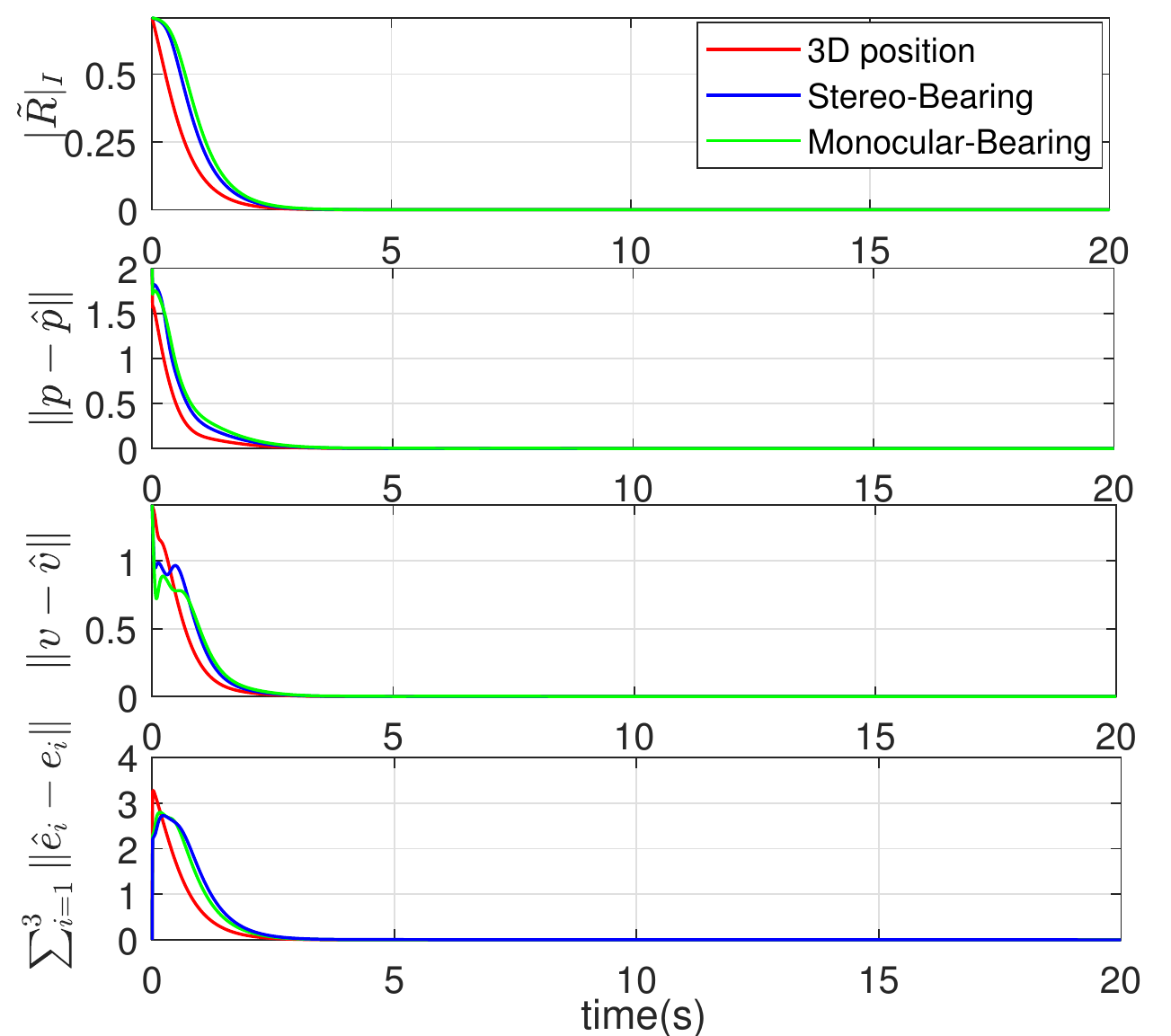}
		\caption{Simulation results of the nonlinear observer \eqref{eqn:observer_1} using 3D landmark position, stereo-bearing and monocular-bearing measurements.}
		\label{fig:simulation_error}
	\end{figure}

	\section{Experimental Results} \label{sec:experimental} 
	To further validate the performance of the proposed observer, two sets of experiments have been presented 
	using the data from the EuRoc dataset \cite{Burri25012016}, where the trajectories are generated by a real flight of a quadrotor. This dataset includes stereo images, IMU measurements, and ground truth obtained from Vicon motion capture system. More details about the EuRoC dataset can be found in \cite{Burri25012016}. Since the sampling rate of the stereo camera (20Hz) is much lower than that of the IMU (200Hz), the proposed hybrid observer \eqref{eqn:observer_2} was implemented.  
	
	The accelerometer and gyro measurements are compensated using the biases provided in the dataset. The features are tracked via the Kanade-Lucas-Tomasi (KLT) tracker \cite{shi1994good}, with the RANSAC  outliers removal algorithm.  Due to the lack of physical landmarks in this dataset, a set of `virtual' landmarks (in the inertial frame) are generated as \cite{wang2020nonlinear}. 
	The IMU measurements are not continuous although obtained at a high rate. The same numerical integration methods in \cite{wang2020nonlinear} for the estimated states are considered. 
	The monocular-bearing measurements are obtained from the right camera.
	
	For comparison purposes, the IEKF developed in \cite{barrau2017invariant} using 3D landmark position measurements  has been considered.  All the observers are concurrently executed using the same set of landmarks and visual measurements.   The initial conditions for the estimated states are given as $\hat{R}(0) = \exp(0.1\pi u^\times)$ with $u\in \mathbb{S}^2$, $\hat{p}(0)=\hat{v}(0) =0_{3\times 1}$ and $\hat{e}_i(0)=e_i, i\in\{1,2,3\}$ and $P(0) = I_{15}$.  The scalar gain parameters are chosen as $k_R=20$,  $\rho_1=0.5,\rho_2=0.3$ and $\rho_3=0.2$.  
	For the CDRE \eqref{eqn:CDRE}, matrices $V(t)$ and $Q(t)$  are tuned using \eqref{eqn:V-Q} (with an additional small identity matrix $0.002I$) with $\Cov( n_x) = \blkdiag(0.0024I_3,0.028I_3)$, $\Cov( n_y) = 0.0005I_{3N}$ for bearing measurements and $\Cov( n_y) = 0.06I_{3N}$ for 3D position measurements. For the IEKF, the gain matrices are tuned as per Section V.B in \cite{barrau2017invariant} using the same above mentioned covariance of the measurements noise.
	The results of the first experiment are shown in Fig. \ref{fig:experimental_error}. The estimates, provided by the proposed observer  and the IEKF converge, after a few seconds, to the vicinity of the ground truth with a nice performance in terms of noise attenuation. The averaged position estimation errors after $10sec$ (\textit{i.e.,} at steady state) are as follows: $3.26$cm for the proposed observer with 3D position measurements, $3.29$cm for the proposed observer with stereo-bearing measurements, $10.99$cm for the proposed observer with monocular-bearing measurements and $2.89$cm for the IEKF with 3D position measurements. The proposed observer using monocular-bearing measurements comes with the largest position estimation error, which is mainly due to the fact that it takes less measurement information and requires stronger conditions (on the motion of the camera and the location of the landmarks) for uniform observability  than the other settings.

	In the second experiment, we consider the scenario where the measurements of the left camera are not available after $120sec$. In this situation, after $120sec$, the 3D landmark positions can not be constructed from a single camera, while the stereo-bearing measurements become the monocular-bearing measurements. The results of the second experiment are shown in Fig. \ref{fig:experimental_error2}. As one can see, the estimates, provided by the proposed observer using 3D landmark position measurements and IEKF, diverge after $120sec$; while the estimates, provided by the proposed observer using stereo-bearing measurements, stay in the vicinity of the ground truth.

	\begin{figure}[!ht]
		\centering
		\includegraphics[width=0.85\linewidth]{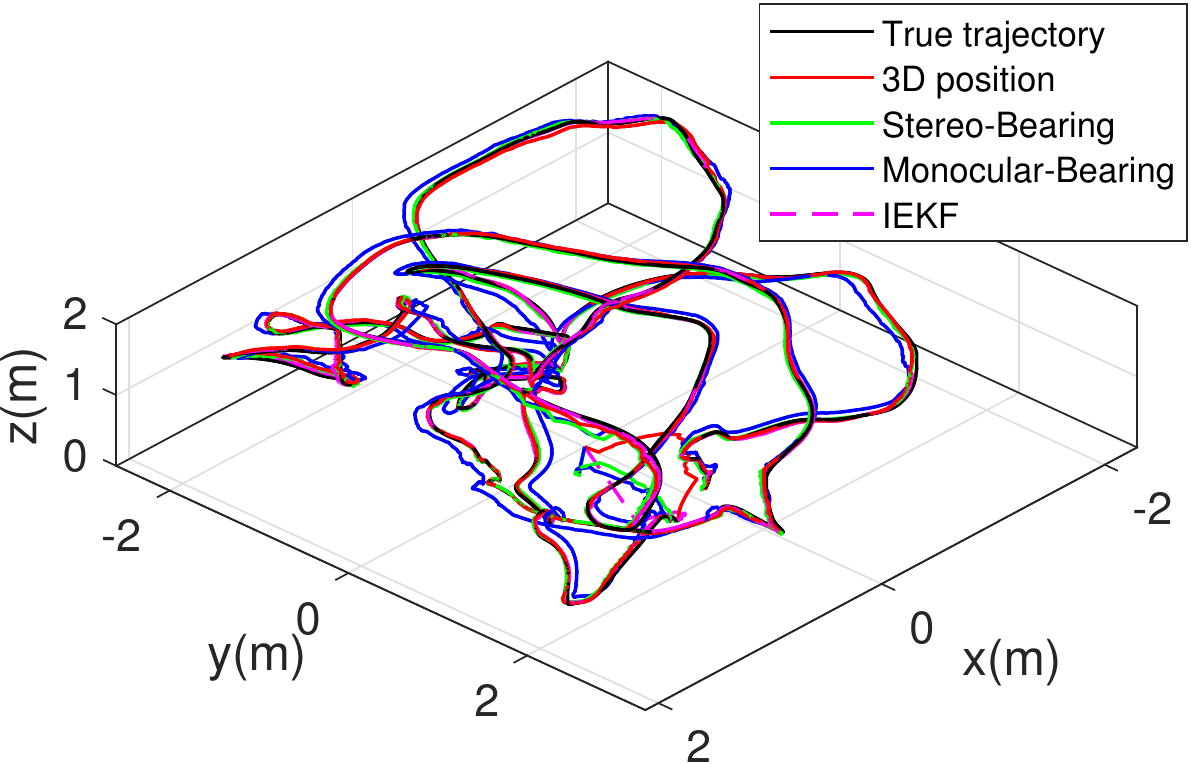}
		\includegraphics[width=0.93\linewidth]{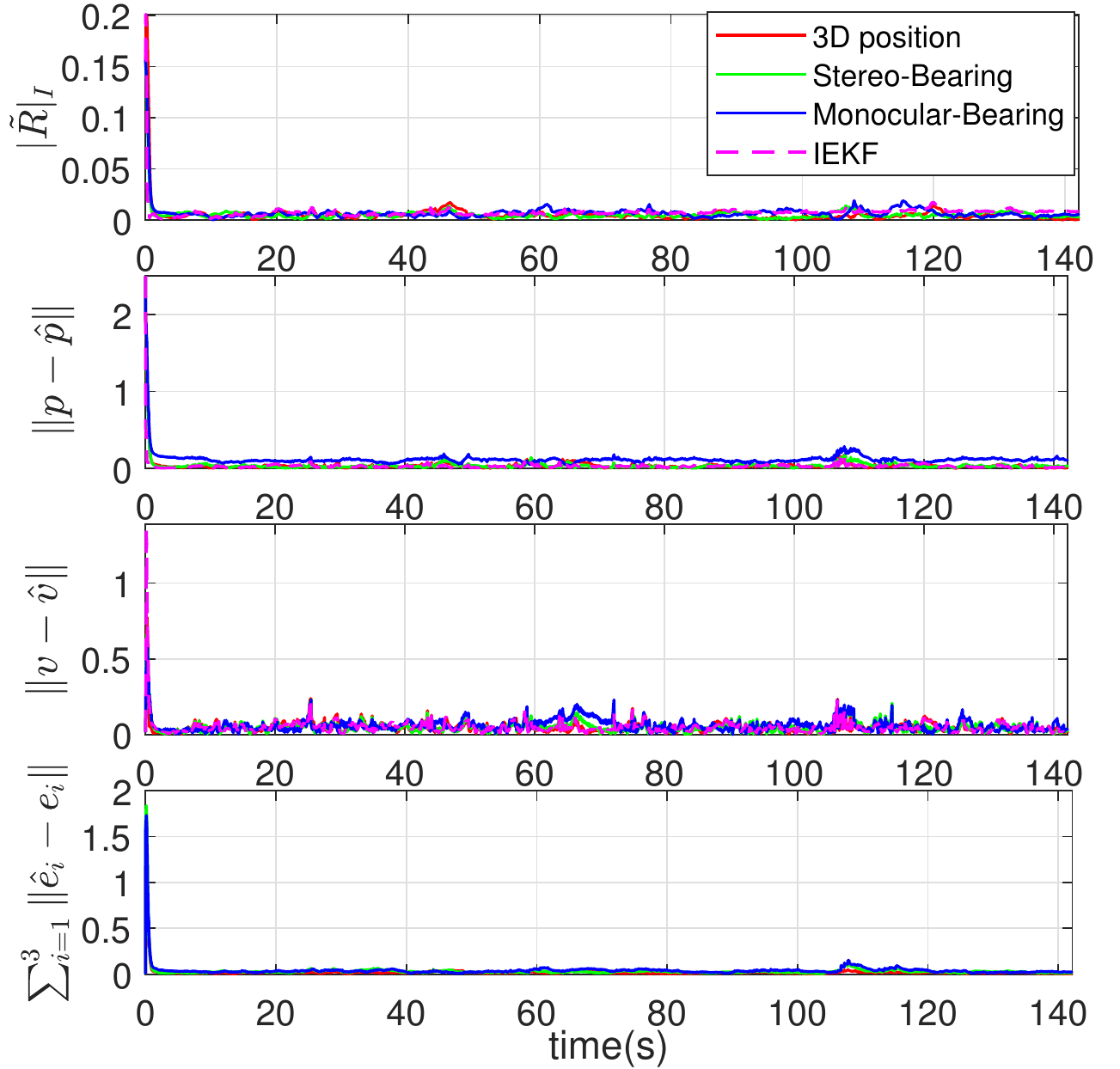}
		\caption{Experimental results  using Vicon Room 1 01 of the EuRoc dataset \cite{Burri25012016}.}
		\label{fig:experimental_error}
	\end{figure} 
	
	\begin{figure}[!ht]
		\centering 
		\includegraphics[width=0.93\linewidth]{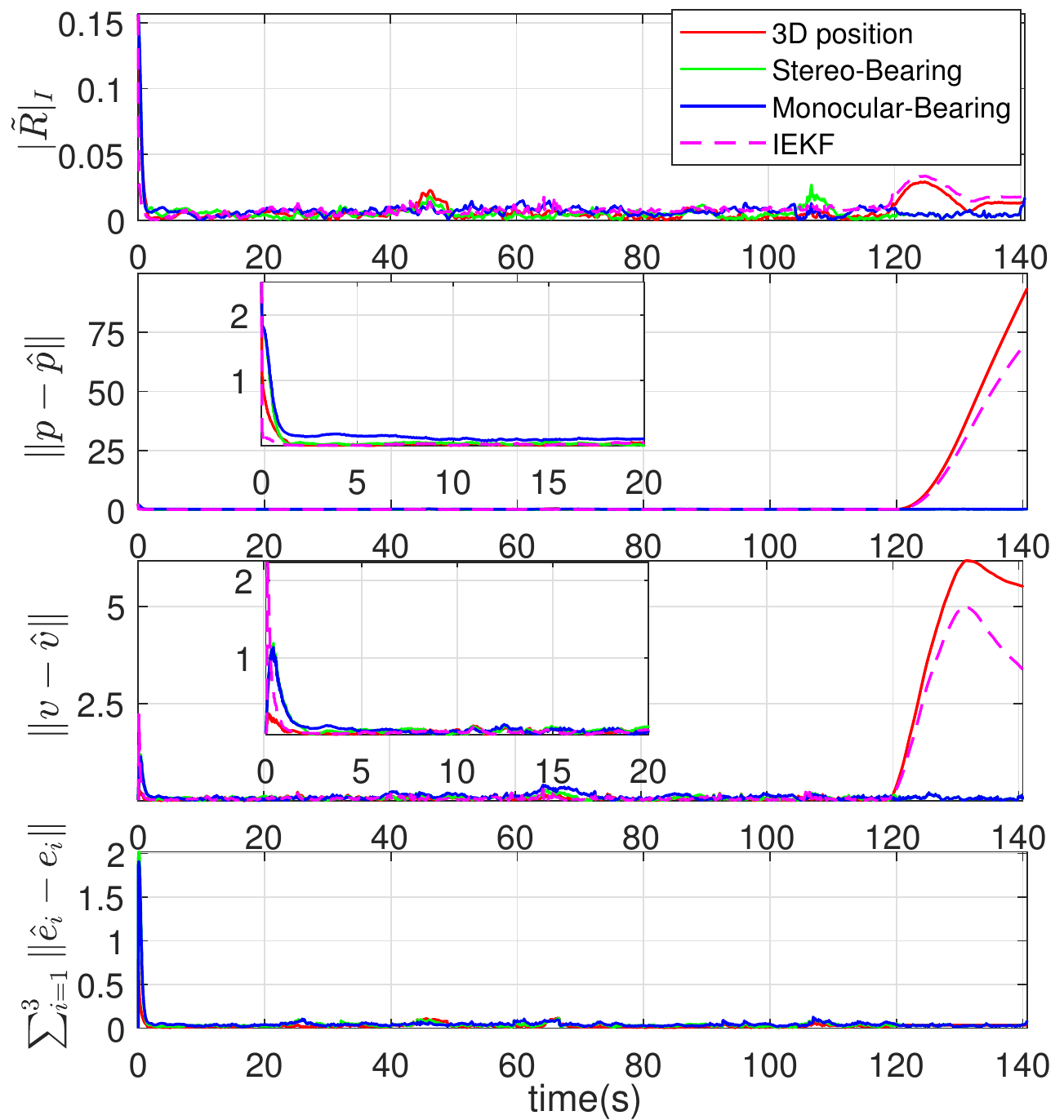}
		\caption{Experimental results   using    Vicon Room 1 01 of the EuRoc dataset \cite{Burri25012016}. The measurements of the left camera are not available after $120sec$.}
		\label{fig:experimental_error2}
	\end{figure}

	\section{Conclusion}
	We addressed the problem of simultaneous estimation of the attitude, position and linear velocity for vision-aided INSs. An AGAS nonlinear observer on $SO(3)\times \mathbb{R}^{15}$ has been proposed using body-frame acceleration and angular velocity measurements, as well as body-frame stereo (or monocular) bearing measurements of some known landmarks in the inertial frame. 
	A detailed uniform observability analysis has been carried out for the monocular and stereo bearing measurements cases. In the stereo bearing measurements case, uniform observability is guaranteed as long as there exist three non-aligned landmarks whose plane is not parallel to the gravity vector. In the monocular bearing case, on top of the condition of the stereo bearing case, it is required that none of the body-frame bearings (of the three non-aligned landmarks whose plane is not orthogonal to the gravity vector) maintains the same direction indefinitely. In the case of a monocular bearing, with a motionless camera, which is known as the static PnP problem, our observer provides a viable solution as long as we have five landmarks that are not in one of the four configurations shown in Fig. \ref{fig:non-observable}.\\
	For practical implementation purposes, we proposed a hybrid version of our nonlinear observer to handle the case where the IMU measurements are continuous and the bearing measurements are intermittently sampled. This observer has been validated using the EuRoc dataset experimental data of a real quadrotor flight. To illustrate the benefit of using bearing measurements over landmark position measurements, we implemented the bearing-based and landmark-position-based observers in a scenario where one of the cameras loses sight of the landmarks after some time, and the results are shown in Fig. \ref{fig:experimental_error2}. 
	As a future work, we intend to  enhance our proposed observer with the estimation of the accelerometer and angular velocity biases while preserving the AGAS property for the overall closed-loop system.

	
	\appendix

	\subsection{Proof of Theorem \ref{theo:theo1}}\label{sec:proof_theo1}
	
	Before proceeding with the proof of Theorem \ref{theo:theo1}, some useful properties on $SO(3)$ are given in the following lemma, whose proof can be found in \cite{berkane2017hybrid,berkane2017phdthesis}.
	\begin{lemma} \label{lemma:LM}
		Consider the trajectory $\dot{R}=R  \omega^\times$ with $R(0)\in SO(3)$ and $\omega\in \mathbb{R}^3$. Let  $\mathcal{L}_M(R)=\tr((I_3-R)M)$ be the the potential function on $SO(3)$ with $M=M\T$   a positive semi-definite matrix. Then, for all $x,y \in \mathbb{R}^3$ the following properties hold:
		\begin{subequations}
			\begin{align}
				4\lambda_{\min}^{\bar{M}} |R|_I^2 & \leq \mathcal{L}_M(R) \leq 4\lambda_{\max}^{\bar{M}}|R|_I^2  \label{eqn:property-tr-MR}\\
				\|\psi_a(MR)\|^2 &= \alpha(M,R)\tr((I_3-R)\underline{M}) \label{eqn:property-psi-mR} \\ 
				\dot{\psi}_a(MR) &= E(MR) \omega  \label{eqn:dotpisR}\\
				\|E(MR)\|_F &\leq \|\bar{M}\|_F \label{eqn:property-EMR2}
			\end{align}
		\end{subequations}
		where $\bar{M} := \frac{1}{2}(\tr(M)I_3-M), \underline{M}: = \tr(\bar{M}^2)I_3 - 2\bar{M}^2$, $E(MR) = \frac{1}{2}(\tr(M {R})I_3-R\T M)$, and the map $\alpha(M,R):=(1-|R|_I^2\cos \measuredangle (u,\bar{M}u))$ with $u\in \mathbb{S}^2$   denoting the axis of the rotation $R$ and $\measuredangle (\cdot, \cdot)$ denoting the angle between two vectors.
	\end{lemma}

	Consider the following real-valued function on $\mathbb{R}^{15}$:
	\begin{equation}
		\mathcal{L}_P(\tilde{x}) = \tilde{x}\T P^{-1} \tilde{x} \label{eqn:L_p}
	\end{equation}
	where $P(t)$ is the solution to the CRE \eqref{eqn:CRE}. Note that $A(t)$ in \eqref{eqn:closed-loop-A}  is continuous and  bounded since $\omega(t)$ is continuous and bounded. Since the pair $(A(t),C(t))$ is uniformly observable, it follows from \cite{bucy1967global,bucy1972riccati} that $p_m I_{15} \leq P(t) \leq p_M I_{15}, \forall t\geq 0$ with some  constants $0< p_m\leq p_M < \infty$. Hence
	\begin{equation}
		\frac{1}{p_M} \|\tilde{x}\|^2 \leq \mathcal{L}_P(\tilde{x}) \leq \frac{1}{p_m} \|\tilde{x}\|^2. \label{eqn:L_p-bound}
	\end{equation}
	From \eqref{eqn:CRE} and \eqref{eqn:dynamic_x}, the time-derivative of $\mathcal{L}_P$ is given by
	\begin{align}
		\dot{\mathcal{L}}_P  &= \tilde{x}\T (P^{-1}A+A\T P^{-1} - 2 C\T Q C + \dot{P}^{-1}) \tilde{x} \nonumber \\
		&= -   \tilde{x}\T P^{-1}V P^{-1} \tilde{x} -  \tilde{x}\T C\T QC \tilde{x}\nonumber \\
		&\leq  -\frac{v_m}{p_M^2}\|\tilde{x}\|^2  \leq - \lambda\mathcal{L}_P
		\label{eqn:dot_L_x} 
	\end{align}
	with $\lambda:= {v_m p_m}/{p_M^2}$, $v_m := \inf_{t\geq 0}\lambda_{\min}^{V(t)}$, where we made use of the facts $C\T QC \geq 0$  and $\dot{P}^{-1}=- P^{-1} \dot{P} P^{-1}=-P^{-1}A-A\T P^{-1}  +  C\T Q C -  P^{-1}V P^{-1}$. Hence, one has $ \|\tilde{x}(t)\| \leq \sqrt{ {p_M}/{p_m}} \exp(-\frac{\lambda}{2} t) \|\tilde{x}(0)\|$, which implies that  $\tilde{x}$ converges to zero exponentially, and $\tilde{x}, \dot{\tilde{x}}$ are bounded.
	Note that the convergence of $\tilde{x}$ is independent from the dynamics of the rotation. From \eqref{eqn:closed-loop},  the equilibrium points of the system are given as $(R^*,0_{15\times 1})$ with $\|\psi_a(MR^*)\|=0$. 
	Using the facts  $\psi_a(MR)=\text{vec} \circ \mathbb{P}_a(MR)$ and $\mathbb{P}_a(MR)=(MR- R\T M)/2$, it follows that $\|\psi_a(M\tilde{R})\|=0$ implies, as shown in \cite{mahony2008nonlinear}, that $\tilde{R}\in \{R\in SO(3):~R=\mathcal{R}_\alpha(\pi,v), v\in \mathcal{E}(M) \}$.

	On the other hand, consider the following real-valued function on $SO(3)$:
	\begin{equation}
		\mathcal{L}_M(\tilde{R}) = \tr((I-\tilde{R})M)  \label{eqn:L_R}
	\end{equation}
	whose time-derivative is given by
	\begin{align}
		\dot{\mathcal{L}}_{{M}} &= \tr(-M \tilde{R} (-k_R \psi_a(M\tilde{R})- \Gamma(t)\tilde{x} )^\times) \nonumber  \\
		& \leq -2k_R\|\psi_R\|^2 +    2c_\Gamma\|\tilde{x}\| \|\psi_R\|   \label{eqn:dot_LM}
	\end{align}
	where $\psi_R: = \psi_a(M\tilde{R})$, and we made use of the facts   $\tr(-A u^\times) = \langle\langle A, u^\times \rangle\rangle = 2u\T \psi_a(A)$ for any $A\in \mathbb{R}^{3\times 3}, u\in \mathbb{R}^3$, and $\|\Gamma(t)\|  \leq   c_\Gamma$ for all $t\geq 0$.  
	
	Consider the following Lyapunov function candidate:
	\begin{align}
		\mathcal{L}(\tilde{R},\tilde{x}) = \mathcal{L}_M(\tilde{R}) + \kappa \mathcal{L}_P(\tilde{x}) \label{eqn:L}
	\end{align}
	with some constant scalar $\kappa>0$. Let  $ \zeta :=[\|\psi_R\|, \|\tilde{x}\|]\T$. From \eqref{eqn:dot_L_x} and \eqref{eqn:dot_LM}, the time-derivative of $\mathcal{L}$ is given by
	\begin{align}
		\dot{\mathcal{L}}  &\leq -2k_R\|\psi_R\|^2 +    2 c_\Gamma\|\tilde{x}\| \|\psi_R\| -\kappa \frac{v_m}{p_M^2} \|\tilde{x}\|^2   \nonumber \\
		& = -\zeta\T H \zeta, \quad H:=\begin{bmatrix}  
			2k_R & - { c_\Gamma} \\
			-  { c_\Gamma}  &  \frac{\kappa v_m}{p_M^2}
		\end{bmatrix}  \label{eqn:dot_L}.
	\end{align}
	Choosing $\kappa> {c_\Gamma^2 p_M^2}/{(2k_R v_m)}$  such that  matrix $H$ is positive definite, it follows that
	$\dot{\mathcal{L}} \leq 0$ and then $\mathcal{L}$ is non-increasing. Then, making use of the facts $\mathcal{L}(\tilde{R}(t),\tilde{x}(t)) - \mathcal{L}(\tilde{R}(0),\tilde{x}(0)) = \int_{0}^{t} \dot{\mathcal{L}} (\tilde{R}(\tau),\tilde{x}(\tau)) d\tau \leq -\int_{0}^{t}  \zeta\T (\tau) H \zeta(\tau) d\tau $, one verifies that $\lim_{t \to \infty}\int_{0}^{t}  \zeta\T (\tau) H \zeta(\tau) d\tau$ exists and is finite. Since $\tilde{x}$ is bounded and matrices $A(t),C(t),Q(t)$ and $P(t)$ are bounded, it is clear that $\dot{\tilde{x}}$ is bounded. From \eqref{eqn:property-tr-MR} and \eqref{eqn:property-psi-mR} in Lemma \ref{lemma:LM}, one obtains that $\psi_R$ is bounded. Moreover, in view of \eqref{eqn:closed-loop} and \eqref{eqn:dotpisR}-\eqref{eqn:property-EMR2}, one can easily verify that $ \dot{\psi}_R$ is bounded. Thus, it follows that the time-derivative of $\zeta\T H \zeta$ is bounded, which implies  the uniform continuity of $\zeta\T H \zeta$. Therefore, by virtue of Barbalat's lemma, one has $\lim_{t \to \infty}  \zeta\T (t) H \zeta(t) = 0$, \ie,  $(\|\psi_R\|,\|\tilde{x}\|)\to (0,0)$ as $t \to \infty$. This implies that, for any initial condition $(\tilde{R}(0),\tilde{x}(0))\in SO(3) \times \mathbb{R}^{15}$, the solution $(\tilde{R},\tilde{x})$ to \eqref{eqn:closed-loop} converges to the set ${(I_3,0_{15\times 1})}\cup \Psi_M$, which proves item (i). 
	
	Next, let us prove the local exponential stability of the equilibrium $(I_3,0_{15\times 1})$ in item (ii). Let $0<\varepsilon_R  < 4\lambda_{\min}^{\bar{M}}$ and define the set $U_{\varepsilon_R} = \{(R, {x})\in SO(3)\times  \mathbb{R}^{15}: \mathcal{L}(R, {x}) \leq \varepsilon_R\}$. From \eqref{eqn:L}-\eqref{eqn:dot_L} with $\kappa>  { c_\Gamma^2 p_M^2}/{(2k_Rv_m)}$, for any initial condition $(\tilde{R}(0),\tilde{x}(0))\in U_{\varepsilon_R}$, one has
	$
	\mathcal{L}_M(\tilde{R}(t)) \leq \mathcal{L}(\tilde{R}(t),\tilde{x}(t)) \leq \mathcal{L}(\tilde{R}(0),\tilde{x}(0)) \leq \varepsilon_R  
	$
	for all $t\geq 0$. It follows from \eqref{eqn:property-tr-MR}-\eqref{eqn:property-psi-mR} that
	\begin{align}
		|\tilde{R}|_I^2 &\leq  {\varepsilon_R/4\lambda_{\min}^{\bar{M}}}< 1 \\
		\varrho |\tilde{R}|_I^2 &\leq \|\psi_R\|^2 \leq 4\lambda_{\max}^{W} |\tilde{R}|_I^2 \label{eqn:psi-R}
	\end{align}
	with matrix $W: = \frac{1}{2}(\tr(\underline{M})I_3- \underline{M})=\bar{M}^2$ and
	$\varrho := \min_{(R,x)\in U_{\varepsilon_R}} 4\alpha(M,R)  \lambda_{\min}^{W} \geq 4(1-\varepsilon_R/4\lambda_{\min}^{\bar{M}} )  \lambda_{\min}^{W} 
	$. Let $\bar{\zeta}:=[|\tilde{R}|_I, \|\tilde{x}\|]\T$. In view of \eqref{eqn:property-tr-MR}, \eqref{eqn:L_p-bound} and \eqref{eqn:L_R}, one obtains  
	\begin{equation}
		\underline{\alpha} \|\bar{\zeta}\|^2 \leq
		\mathcal{L}
		\leq \bar{\alpha} \|\bar{\zeta}\|^2   \label{eqn:L_zeta_bar}
	\end{equation}
	where $\underline{\alpha} := \min\{4\lambda_{\min}^{\bar{M}},  \frac{\kappa}{p_M}\}$ and  $\bar{\alpha} := \max\{4\lambda_{\max}^{\bar{M}},  \frac{\kappa}{p_m}\}$. Substituting  \eqref{eqn:psi-R} into \eqref{eqn:dot_L}, one has
	\begin{align}
		&\dot{\mathcal{L}}
		\leq -2k_R\varrho |\tilde{R}|_I^2 +     4 \sqrt{\lambda_{\max}^{W}} c_\Gamma\|\tilde{x}\| |R|_I - \frac{\kappa v_m}{p_M^2}\|\tilde{x}\|^2  \nonumber \\
		&~~~\leq    -\bar{\zeta}\T {\bar{H}}\bar{\zeta},  ~  \bar{H}:=\begin{bmatrix}
			2 k_R \varrho  & -  2  \sqrt{\lambda_{\max}^{W}} c_\Gamma \\
			- 2    \sqrt{\lambda_{\max}^{W}} c_\Gamma  & \frac{\kappa v_m}{p_M^2}	
		\end{bmatrix}.  \label{eqn:dot_L2}
	\end{align}
	Choosing $\kappa >  { 2  \lambda_{\max}^{W} c_\Gamma^2 p_M^2}/{(k_R \varrho v_m)} $, one can show that both matrices  $H$ and $\bar{H}$ are positive definite since $\varrho \leq  4\lambda_{\min}^{W} \leq 4\lambda_{\max}^{W}$.  In view of \eqref{eqn:L_zeta_bar} and \eqref{eqn:dot_L2}, one concludes
	$
	\|\bar{\zeta}(t)\|  \leq  \sqrt{ {\bar{\alpha}}/{\underline{\alpha}}} \exp(-\frac{1}{2}\lambda_{\min}^{\bar{H}} t) \|\bar{\zeta}(0)\| 
	$
	for all $t\geq 0$, which implies that $(\tilde{R},\tilde{x})$ converges to $(I_3,0_{15\times 1})$ exponentially for any initial condition $(\tilde{R}(0),\tilde{x}(0))\in U_{\varepsilon_R}$. This completes the proof of item (ii).

	Now, we need to show that the undesired equilibria of \eqref{eqn:closed-loop} defined by the set $\Psi_M$ are unstable. From \eqref{eqn:dot_L_x}, one shows that $\tilde{x}$ converges to zero exponentially, and is independent from the dynamics of $\tilde{R}$. Then, we focus on the dynamics of \eqref{eqn:closed-loop} at $\|\tilde{x}\|=0$. For each $v\in \mathcal{E}(M)$, let us define $R^*_v=\mathcal{R}_\alpha(\pi,v)$  and the open set  $U_v^\delta:=\{(R,x)\in SO(3)\times \mathbb{R}^{15}: R = R^*_v \exp(\epsilon^\times),  \|\epsilon\|  \leq \delta, \|x\|=0\}$ with $\delta$ sufficiently small. For any $(\tilde{R},\tilde{x})\in U_v^\delta$, pick a sufficiently small $\epsilon$ such that $(R^*_v)\T \tilde{R}:=\exp(\epsilon^\times) \approx   I_3 + \epsilon^\times  $. Consequently, from  \eqref{eqn:dynamic_R}  one  obtains the linearized dynamics of $\epsilon$ around  $R^*_v$  as follows:
	\begin{align}
		\dot{\epsilon} &   =- k_R W_v \epsilon  \label{eqn:dot_epsion}
	\end{align}
	where $\psi_a(k_R MR^*_v(I_3+\epsilon^\times)) = k_R W_v \epsilon$ with $W_v = W_v\T := \frac{1}{2} (\tr(MR^*_v)I_3-(MR^*_v)\T)= \frac{1}{2}(2v\T M v-\tr(M))I_3-  \frac{1}{2}(2 \lambda_v^M vv\T   - M)$, and we made use of the facts $M v = \lambda_v^M v$,  $\psi_a(MR^*_v)=0, \psi_a(MR^*_v \epsilon^\times) =\text{vec} \circ \mathbb{P}_a(MR^*_v \epsilon^\times)$ and $\mathbb{P}_a(MR^*_v \epsilon^\times)= \frac{1}{2} (MR^*_v \epsilon^\times + \epsilon^\times (MR^*_v)\T) =  (W_v \epsilon)^\times $.  
	Since $W_v = W_v\T$ and $M$ is positive semi-definite with three distinct eigenvalues, one verifies that $-2v\T W_v v = -v\T M v+\tr(M)>0$, which implies that   $-W_v, \forall v\in \mathcal{E}(M)$ has at least one positive eigenvalue. Then, one can conclude that all the equilibrium points in $\Psi_M$ are unstable.  Consider the subsystem $\dot{\tilde{R}} = \tilde{R}(-k_R\psi_a(M\tilde{R}))^\times$ and its   linearized dynamics around the undesired   equilibrium points $R^*_v=\mathcal{R}_\alpha(\pi,v), v\in \mathcal{E}(M)$ in \eqref{eqn:dot_epsion}. 
	For each undesired equilibrium point $R^*_v$, there exist (local) stable and unstable manifolds, and the union of the stable manifolds and the undesired equilibria has dimension less than three \cite[The Stable Manifold Theorem]{perko2013differential}. Then, the set of  the union of the stable manifolds and the undesired equilibria has Lebesgue measure zero.  It follows that the solution $\tilde{R}(t)$ converges to $I_3$ starting from all initial conditions except from a set of Lebesgue measure zero. Hence, one   concludes that the   equilibrium $(I_3, 0_{15\times 1})$ of the overall system  \eqref{eqn:closed-loop} is almost globally asymptotically stable. This completes the proof.

	\subsection{Proof of Lemma \ref{lemma:lemmaUOC}} \label{sec:proof_lemmaUOC} 	
	The proof of this Lemma is motivated from \cite[Lemma 3.1]{scandaroli2013visuo} and \cite[Lemma 2.7]{hamel2017position}. 	
	In order to show that the pair $(A,C(t))$ is uniformly observable, we are going to verify the existence  of constants $\delta,\mu>0$ such that  
	$ z\T W_o(t,t+\delta)z =   \frac{1}{\delta}  \int_{t}^{t+\delta} \|C(\tau) \Phi(\tau,t) z\|^2 d\tau \geq \mu$ for all $ t\geq 0$ and $z\in \mathbb{S}^{n-1}$. 
	Let us proceed by contradiction and assume that the pair $(A, C(t))$ is not uniformly observable, \textit{i.e.,}  
	\begin{align}
		\forall \bar{\delta},\bar{\mu}>0, \exists t\geq 0,  \min_{z\in \mathbb{S}^{n-1}}\frac{1}{\bar{\delta}}  \int_{t}^{t+\bar{\delta}} \|C(\tau) \Phi(\tau,t) z\|^2 d\tau < \bar{\mu}.  \label{eqn:contradiction0}
	\end{align}
	Consider a sequence $\{\mu_q\}_{q\in \mathbb{N}}$ of positive numbers converging to zero, and an arbitrary positive scalar $\bar{\delta}$. Then, there must exist a sequence of time instants $\{t_q\}_{q\in \mathbb{N}}$ and a sequence of   vectors $\{z_q\}_{q\in \mathbb{N}}$ with $z_q\in \mathbb{S}^{n-1}$ such that  $\frac{1}{\bar{\delta}}  \int_{t_q}^{t_q+\bar{\delta}} \|C(\tau) \Phi(\tau,t_q) z_q\|^2 d\tau < \bar{\mu}_q$  for any $q\in \mathbb{N}$. Since the set $\mathbb{S}^{n-1}$ is compact, there exists a sub-sequence of $\{z_q\}_{q\in \mathbb{N}}$ which converges to a limit $\bar{z}\in \mathbb{S}^{n-1}$. Moreover, since   $C(t)$ is bounded and the interval
	of integration in \eqref{eqn:contradiction0} is fixed,  it follows from \eqref{eqn:contradiction0} that
	\begin{equation}
		\lim_{q\to  \infty}  \int_{0}^{\bar{\delta}} \|C(t_q + \tau)  \Phi(t_q + \tau,t_q) \bar{z}\|^2 d\tau =0  \label{eqn:contradiction}
	\end{equation}	
	by a change of integration variable. 
	Using the facts $A=S+N$ and  $SN=NS$,  the state transition matrix $\Phi(t_q + \tau,t_q)$ can be explicitly expressed as   $\Phi(t_q + \tau,t_q) =\exp(A\tau)=\exp(S\tau) \exp(N\tau) $. Then,  \eqref{eqn:contradiction} is equivalent to $	\lim_{q\to  \infty}   \int_{0}^{\bar{\delta}} \left\|C(t_q + \tau)  \exp(S\tau) \exp(N\tau)\bar{z}  \right\|^2  d\tau =  0$, which implies
	\begin{align}
		\lim_{q\to  \infty}   \int_{\bar{\delta}-\delta}^{\bar{\delta}}\left\|C(t_q + \tau)  \exp(S\tau) \exp(N\tau)\bar{z}  \right\|^2  d\tau =  0  \label{eqn:contradiction1} 
	\end{align}  
	with some $0<\delta<\bar{\delta}$.
	Consider now the following technical results   whose proofs are given after this proof. 	
	\begin{lemma}\label{lemma:CexpN} 
		From \eqref{eqn:contradiction1} with  $\bar{\delta}$ large enough, one has
		\begin{equation}
			\lim_{q\to  \infty} \int_{\bar{\delta}-\delta}^{\bar{\delta}}     \left\|C(t_q + \tau)  \exp(N\tau)\bar{z}\right\|  d\tau =0 .  \label{eqn:contradiction2}
		\end{equation}   
	\end{lemma}	
	\begin{lemma}\label{lemma:CN} 
		From \eqref{eqn:contradiction2} with  $\bar{\delta}$ large enough, one has
		\begin{equation}
			\lim_{q\to  \infty} \int_{\bar{\delta}-\delta}^{\bar{\delta}}   \|C(t_q + \tau)   N^k \bar{z}\|^2 d\tau =0   \label{eqn:contradiction3}
		\end{equation}
		for all $k = 0,1,\dots,s-1$.
	\end{lemma}	
	Using the fact that the matrix  $\mathcal{O}(t)$ is composed of row vectors   of   $C(t) $, $C(t) N, \dots, C(t) N^{s-1}$, one has $ \sum_{k=0}^{s-1} \|C(t_q + \tau)  N^k \bar{z}\|^2    \geq   \|O(t_q + \tau) \bar{z}\|^2    	   
	$ for every $q\in \mathbb{N}_{>0}$.
	It follows from \eqref{eqn:contradiction2}  and \eqref{eqn:contradiction3}   that
	\begin{align}
		& \lim_{q\to  \infty} \int_{t_q +\bar{\delta}-\delta }^{t_q + \bar{\delta}} \|\mathcal{O}(\tau ) \bar{z}\|^2 d\tau   =  \lim_{q\to  \infty} \int_{\bar{\delta}-\delta}^{\bar{\delta}}   \|\mathcal{O}(t_q + \tau)  \bar{z}\|^2 d\tau \nonumber \\  
		&\qquad   \leq  \lim_{q\to  \infty} \int_{\bar{\delta}-\delta}^{\bar{\delta}}  \sum\nolimits_{k=0}^{s-1} \|C(t_q  + \tau)  N^k \bar{z}\|^2 d\tau  =0.  \label{eqn:Qz}
	\end{align}	
	On the other hand, in view of \eqref{eqn:gramiancondition1},  one can show that
	$
	\int_{t_q +\bar{\delta}-\delta }^{t_q + \bar{\delta}} \|\mathcal{O}(\tau ) \bar{z}\|^2 d\tau    = \bar{z}\T  ( \int_{t_q +\bar{\delta}-\delta }^{t_q + \bar{\delta}} \mathcal{O}\T(\tau)  \mathcal{O}(\tau)   d\tau  ) \bar{z} > \mu$ for each $ \bar{z} \in \mathbb{S}^{n-1} 
	$,
	which contradicts \eqref{eqn:Qz}. It implies that \eqref{eqn:contradiction3} is not true for all $k=0,1,\dots,s-1$, and in turns   \eqref{eqn:contradiction} and \eqref{eqn:contradiction1} are not true when $\bar{\delta}$ is large enough. Therefore, one can always find $\bar{\delta}$   large enough, such that \eqref{eqn:contradiction0} does not hold. Consequently, one   concludes that the pair $(A,C(t))$ is uniformly observable. It remains to prove Lemma \ref{lemma:CexpN} and   \ref{lemma:CN}.	 
	\subsubsection{Proof of Lemma \ref{lemma:CexpN}}
	We are going to show that \eqref{eqn:contradiction1} implies \eqref{eqn:contradiction2}
	provided that $\bar{\delta}$ is large enough. From \cite[Theorem 1]{perko2013differential} there exist an invertible matrix $P$ and a diagonal matrix $D$ such that   $D= \diag(\lambda_1,\cdots,\lambda_n)=P^{-1}SP$   with  $\lambda_1,\cdots,\lambda_n$ denoting the eigenvalues of $A$ repeated according to their multiplicity. Then, one obtains $\exp(S\tau) = P\exp(D\tau)P^{-1}$. 
	Let $\bar{z}'(\tau) =  P^{-1} \exp(N\tau) \bar{z}$   and rewrite the eigenvalues of $D$ as $\lambda'_1< \cdots <\lambda'_d$ with  $d\leq n$ denoting the number of distinct eigenvalues. It follows that $\exp(D\tau)  P^{-1} \exp(N\tau) \bar{z} = \exp(D\tau)  \bar{z}'(\tau) = \sum_{i=1}^d \exp(\lambda_i' \tau) \bar{z}_{i}'(\tau)$ with $\bar{z}_{i}'(\tau) \in \mathbb{R}^{15}$ for all $ i=1, \dots,d$ and  $\sum_{i=1}^d \bar{z}_{i}'(\tau) =\bar{z}'(\tau) $.  We assume that there exist a constant $\epsilon'>0$ and a sub-sequence of $\{i\}_{1\leq i \leq d} \subset \mathbb{N}_{>0}$  such that
	\begin{equation}
		\lim_{q\to  \infty}   \int_{\bar{\delta}-\delta}^{\bar{\delta}}  \|C(t_q + \tau)  P \bar{z}_i'(\tau)\|  d\tau > \epsilon'. \label{eqn:contradiction1-j2}
	\end{equation}	
	Then, one can always pick the largest $j$ from the sub-sequence. For the sake of simplicity, let us define $\nu_j(\tau)=\sum_{i=1}^j \exp(\lambda_i' \tau) \bar{z}_{i}'(\tau) $.
	Using the facts  $\| C(t_q + \tau)   P \nu_j(\tau)\|  \leq \| C(t_q + \tau)  P \nu_{d}(\tau) \| +  \sum_{i=j+1}^d\exp(\lambda_i \tau)  \|C(t_q + \tau)  P\bar{z}_{i}'(\tau) \|  $, $P\nu_{d}(\tau) = P\exp(D\tau)  P^{-1} \exp(N\tau) \bar{z} =\exp(S\tau)  \exp(N\tau) \bar{z} $ and	$  \lim_{q\to  \infty}  \int_{\bar{\delta}-\delta}^{\bar{\delta}}  \|C(t_q + \tau)   P\bar{z}_{i}'(\tau)\| = 0$ for all $j+1 \leq i\leq d$, from \eqref{eqn:contradiction1}    one can   show that
	\begin{align}
		&\lim_{q\to  \infty}   \int_{\bar{\delta}-\delta}^{\bar{\delta}} \left\|C(t_q + \tau)  P \nu_{j}(\tau) \right\|  d\tau  \nonumber \\			
		&   \leq \sum_{i=j+1}^d\exp(|\lambda_i| \bar{\delta}) \lim_{q\to  \infty}   \int_{\bar{\delta}-\delta}^{\bar{\delta}}  \|C(t_q + \tau)  P\bar{z}_{i}'(\tau) \| d\tau \nonumber \\
		& ~ +  \lim_{q\to  \infty}   \int_{\bar{\delta}-\delta}^{\bar{\delta}}  \| C(t_q + \tau)    \exp(S\tau)  \exp(N\tau) \bar{z}  \| d\tau  =    0. \label{eqn:contradiction1-j} 
	\end{align}		 
	Define  $\eta'(\tau) := \sum_{i=1}^{j-1}   \exp((\lambda_i'-\lambda_j') \tau) \bar{z}_{i}'(\tau)$  such that one has $ \exp(-\lambda_j' \tau)\nu_{j}(\tau) = \bar{z}_{j}'(\tau) + \eta'(\tau)$. Then, using the facts $ \sum_{i=1}^d \bar{z}_{i}'(\tau)=\bar{z}'(t)=P^{-1} \exp(N\tau) \bar{z} = P^{-1} \sum_{k=0}^{s-1}N^k \bar{z}$ and $\lim_{t\to \infty}\exp(-at)t^b =0$ for any $a,b>0$,  one obtains $\eta'(\tau) \to 0$ as $\tau \to \infty$.  Since  $C(t)$  is continuous and   bounded by assumption, there exists a positive constant $\gamma'$ given as $\gamma' = \sup_{t\geq 0} \|C(t)P\|$. Thus, from \eqref{eqn:contradiction1-j2}   one can show that 
	\begin{align*}
		& \exp(|\lambda_j'| \bar{\delta}) \int_{\bar{\delta}-\delta}^{\bar{\delta}} \left\|C(t_q + \tau)   P \nu_j(\tau) \right\|  d\tau   \nonumber \\ 
		& ~~\geq \int_{\bar{\delta}-\delta}^{\bar{\delta}} \left\|C(t_q + \tau)  P \exp(-\lambda_j' \tau) \nu_j(\tau) \right\|  d\tau   \nonumber \\ 
		& ~~\geq   \int_{\bar{\delta}-\delta}^{\bar{\delta}} \left\|C(t_q + \tau)   P ( \bar{z}_{j}'(\tau) + \eta'(\tau))\right\|  d\tau \nonumber \\
		&~~ \geq   \int_{\bar{\delta}-\delta}^{\bar{\delta}} \left\|C(t_q + \tau)  P   \bar{z}_{j}'(\tau)  \right\|  d\tau  - \gamma'  \int_{\bar{\delta}-\delta}^{\bar{\delta}} \left\|  \eta'(\tau) \right\|  d\tau
	\end{align*}
	Choosing $\bar{\delta}$ large enough such that $\sup_{\tau\in [\bar{\delta}-\delta,\bar{\delta}]} \|\eta'(\tau)\| \leq  {\frac{\epsilon'}{2\delta \gamma' }}$, it follows that
	$
	\lim_{q\to  \infty} \int_{\bar{\delta}-\delta}^{\bar{\delta}} \left\|C(t_q + \tau)   P \nu_j(\tau) \right\|  d\tau   \geq  \frac{\epsilon'}{2} \exp(-|\lambda_j'| \bar{\delta})   > 0 , 
	$
	which contradicts \eqref{eqn:contradiction1-j}.  It means that the sub-sequence satisfying \eqref{eqn:contradiction1-j2} does not exist, \ie, the assumption according to \eqref{eqn:contradiction1-j2} does not hold. Therefore, one obtains $\lim_{q\to  \infty}   \int_{\bar{\delta}-\delta}^{\bar{\delta}}  \|C(t_q + \tau)  P \bar{z}_i'(\tau)\|  d\tau =0$ for all $ i=1,2,\dots,d$. Using the fact $\exp(N\tau) \bar{z}=P\sum_{i=1}^d \bar{z}_{i}'(\tau)$, one can show that $\lim_{q\to  \infty}   \int_{\bar{\delta}-\delta}^{\bar{\delta}} \|C(t_q + \tau)  \exp(N\tau)\bar{z} \| \leq  \sum_{i=1}^d  \lim_{q\to  \infty}   \int_{\bar{\delta}-\delta}^{\bar{\delta}} \|C(t_q + \tau)  P\bar{z}_{i}'(\tau) \| =0 $, which gives \eqref{eqn:contradiction2}.

	\subsubsection{Proof of Lemma \ref{lemma:CN}}
	We are going to show \eqref{eqn:contradiction3} for all $k=0,1,\dots,s-1$ from \eqref{eqn:contradiction2} with some $\bar{\delta}$ large enough.
	Let us proceed by contradiction and assume that \eqref{eqn:contradiction3} does not hold for any $0\leq k \leq s-1$, \ie, there exist a constant $\epsilon>0$ and a sub-sequence of $\{k\}_{0\leq k \leq s-1} \subset \mathbb{N}$  such that
	\begin{equation}
		\lim_{q\to  \infty} \int_{\bar{\delta}-\delta}^{\bar{\delta}}   \|C(t_q + \tau)   N^k \bar{z}\| d\tau > \epsilon. \label{eqn:contradiction_k}
	\end{equation}	
	One can always pick the largest $\bar{k}$ such that \eqref{eqn:contradiction_k} holds.
	Since $N$ is a nilpotent matrix of order $s\leq n$, one has $\exp(N\tau) = \sum_{k=0}^{s-1}\frac{\tau^{ {k}}}{ {k}!}N^k$. For the sake of simplicity, define $\varSigma_k(\tau) = \sum_{i=0}^{k}\frac{\tau^{i}}{i!}N^i$ with   $k\in \mathbb{N}$. 
	Using the fact $\|C(t_q + \tau)  \varSigma_{\bar{k}}(\tau)\bar{z}\| \leq  \left\|C(t_q + \tau)  \exp(N\tau) \bar{z}\right\|  + \sum_{i=\bar{k}+1}^{s-1} \frac{\tau^{i}}{i!}   \|C(t_q + \tau)   N^{i} \bar{z} \| $ and $\lim_{q\to  \infty}   \int_{\bar{\delta}-\delta}^{\bar{\delta}}\|C(t_q + \tau)   N^{i} \bar{z}\| d\tau =0$ for all $\bar{k}+1\leq i \leq s-1$, from \eqref{eqn:contradiction2}  one can show that
	\begin{align}
		&\lim_{q\to  \infty} \int_{\bar{\delta}-\delta}^{\bar{\delta}}   \|C(t_q + \tau)  \varSigma_{\bar{k}}(\tau)\bar{z}\| d\tau \nonumber \\
		&~~~ \leq   \lim_{q\to  \infty}  \int_{\bar{\delta}-\delta}^{\bar{\delta}}  \left\|C(t_q + \tau)\exp(N\tau)   \bar{z}\right\| d\tau \nonumber \\
		& \quad     + \lim_{q\to  \infty} \sum\nolimits_{i=\bar{k}+1}^{s-1} \frac{\bar{\delta}^{i}}{ i! }     \int_{\bar{\delta}-\delta}^{\bar{\delta}}\|C(t_q + \tau)   N^{i} \bar{z}\| d\tau =0. \label{eqn:contradiction2-k}
	\end{align} 	 
	Let  $\eta(\tau)= \frac{\bar{k}!}{\tau^{\bar{k}}} \varSigma_{\bar{k}-1}(\tau)\bar{z}$ such that $ \frac{\bar{k}!}{\tau^{\bar{k}}}\varSigma_{\bar{k}}(\tau)\bar{z} =N^{\bar{k}}\bar{z}+\eta(\tau)$. It is easy to show that $\eta(\tau)$ is bounded and $\lim_{t\to\infty}\eta(\tau) = 0$. Since  $C(t)$  is bounded by assumption, there exist a positive constant $\gamma$ such that $\gamma = \sup_{t\geq 0} \|C(t) \|$.   Then,  one obtains
	\begin{align*}
		&\frac{\bar{k}! }{(\bar{\delta}-\delta)^{k}}   \int_{\bar{\delta}-\delta}^{\bar{\delta}}  \|C(t_q + \tau)  \varSigma_{\bar{k}}(\tau)   \bar{z}  \|  d\tau  \nonumber \\
		& ~~~ \geq  \int_{\bar{\delta}-\delta}^{\bar{\delta}}  \left\|  C(t_q + \tau)   \frac{\bar{k}!}{\tau^{\bar{k}}}\varSigma_{\bar{k}}(\tau)\bar{z} \right\| d\tau \nonumber \\
		&~~~ =  \int_{\bar{\delta}-\delta}^{\bar{\delta}} \|  C(t_q + \tau)  (N^{\bar{k}}\bar{z}+ \eta(\tau))\| d\tau \nonumber \\ 
		&~~~ \geq  \int_{\bar{\delta}-\delta}^{\bar{\delta}}   \|C(t_q + \tau)  N^{\bar{k}}\bar{z}\| d\tau   - \gamma  \int_{\bar{\delta}-\delta}^{\bar{\delta}} \|  \eta(\tau)\| d\tau.   
	\end{align*}  
	Choosing $\bar{\delta}$ large enough such that $\sup_{\tau\in [\bar{\delta}-\delta,\bar{\delta}]}\|\eta(\tau)\|\leq   {\frac{\epsilon}{2\delta \gamma }}$,   from \eqref{eqn:contradiction_k} one can show that   
	$
	\lim_{q\to  \infty} \int_{\bar{\delta}-\delta}^{\bar{\delta}} \|C(t_q + \tau) \varSigma_{\bar{k}}(\tau) \bar{z}\|  d\tau   
	\geq  {\epsilon(\bar{\delta}-\delta)^{\bar{k}}}/{(2\bar{k}!)} >0 ,  
	$
	which contradicts \eqref{eqn:contradiction2-k}, which means that the sub-sequence satisfying \eqref{eqn:contradiction_k} does not exist, \ie,  assumption \eqref{eqn:contradiction_k} does not hold. Thus,  one can show that 
	$
	\lim_{q\to  \infty} \int_{0}^{{\delta}} \|C(t_q + \bar{\delta}-\delta+ \tau)  N^{k} \bar{z}\| d\tau =0  \label{eqn:contradiction2_k}  
	$
	for all $k=0,1,\dots,s-1$. 
	Therefore,   one   can conclude that   \eqref {eqn:contradiction3}  holds for all $k=0,1,\dots,s-1$ with some $\bar{\delta}$ large enough. This completes the proof.

	\subsection{Proof of Lemma \ref{lemma:AC2}}\label{sec:lemmaAC2}
	From the definition of $A(t)$ in \eqref{eqn:closed-loop-A},  one can rewrite $A(t)=\bar{A} + S(t)$ with a block diagonal skew symmetric matrix  $S(t)= \blkdiag(-\omega^\times,-\omega^\times, \dots,-\omega^\times)\in \mathbb{R}^{15 \times 15}$ and a constant matrix $\bar{A}$ such that   $\bar{A}S(t)= S(t)\bar{A}$.
	Let us introduce the following block diagonal matrix: 
	\begin{align}
		T(t) = \blkdiag( {R}\T,  {R}\T,\dots, {R}\T)\in \mathbb{R}^{15 \times 15} \label{eqn:def_T}
	\end{align} 
	whose dynamics are given as $\dot{T}(t) = S(t)T(t)$. One can verify that $T(t)  T\T(t) = I_{15}$ and $T(t)\bar{A} = \bar{A}T(t)$. Let $\bar{\Phi}(t,\tau) = \exp(\bar{A}(t-\tau))$ be the state transition matrix associated to $\bar{A}$. Using similar steps as in the proof of \cite[Lemma 3]{wang2020hybrid}, the state transition matrix associated to $A(t)$ can be  written as 
	\begin{align}
		\Phi(t,\tau) = T(t)\bar{\Phi}(t,\tau) T^{-1}(\tau) \label{eqn:def_Phi}
	\end{align} 
	with the properties:  $\frac{d}{dt} \Phi(t,\tau)= A(t) \Phi(t,\tau)$,  $\Phi(t,t) = I$,  $\Phi^{-1}(t,\tau) = \Phi(\tau,t)$ for all $t,\tau\geq 0$, and   $ \Phi(t_3,t_2)\Phi(t_2,t_1) = \Phi(t_3,t_1)$ for every $t_1,t_2,t_3\geq 0$. Moreover, one can rewrite $C(t)$ as $C(t)=\Theta(t)\bar{C}$ with     $\Theta(t) = \blkdiag( {\Pi}_1(t),\dots,  {\Pi}_N(t))\in \mathbb{R}^{3N\times 3N}$ and $\bar{C}$ defined in \eqref{eqn:Cbar}.
	Define a new block matrix 
	\begin{equation}
		\bar{T}(t) = \blkdiag( {R}\T,  {R}\T,\dots, {R}\T)\in \mathbb{R}^{3N\times 3N}   \label{eqn:def_barT}.
	\end{equation}
	One can show that  $\bar{C} T(t) =  \bar{T}(t) \bar{C}$ and $\bar{T}\T (t) \bar{T}(t) = I_{3N}$. From \eqref{eqn:gramiancondition} and \eqref{eqn:def_T}-\eqref{eqn:def_barT}, one can show that
	\begin{align}
		W_o(t,t+\delta)  
		&  =\frac{1}{\delta} \int_{t}^{t+\delta} \Phi\T(\tau,t) \bar{C}\T \Theta\T(\tau) \Theta(\tau) \bar{C} \Phi(\tau,t) d\tau   \nonumber \\
		&  \geq   T(t) W'_o(t,t+\delta)  T^{-1} (t)  \label{eqn:W_oA_C2}
	\end{align} 
	where $ W'_o(t,t+\delta) =\frac{1}{\delta}\int_{t}^{t+\delta} \bar{\Phi}\T(\tau,t)    \bar{C}\T \bar{\Theta}\T(\tau)   \bar{\Theta}(\tau)  \bar{C}   \bar{\Phi}(\tau,t) \\ d\tau$ with $\bar{\Theta}(t) = \bar{T}\T(t)\Theta(t)\bar{T}(t) =  \blkdiag(\bar{\Pi}_1,\dots, \bar{\Pi}_N)$ and $\bar{\Pi}_i = R\Pi_i R\T$ for each $i=1,2,\dots,N$.  
	
	Next, we are going to show that the pair $(\bar{A},\bar{\Theta}(t)  \bar{C})$ is  uniformly observable. Since $\bar{A}$ is nilpotent with $\bar{A}^3 = 0$, from Lemma \ref{lemma:lemmaUOC}, the pair $(\bar{A},\bar{\Theta}(t)  \bar{C})$ is uniformly observable if there exist scalars $\delta,\mu>0$ such that condition \eqref{eqn:gramiancondition1} holds 
	with $\mathcal{O}(t)  = [(\bar{\Theta}(t)\bar{C})\T, (\bar{\Theta}(t)\bar{C}\bar{A})\T, (\bar{\Theta}(t)\bar{C}\bar{A}^2)\T]\T$.
	Since matrix $\Pi_i$ is uniformly positive definite
	for each $i=1,2,\dots,N$, one can show that $\bar{\Pi}_i, i\in\{1,2,\dots,N\}$ and $\bar{\Theta}(t)$ are uniformly positive definite.  Then, one has $\text{rank}(\mathcal{O})  = \text{rank}([ \bar{C}\T, (\bar{C}\bar{A})\T, (\bar{C}\bar{A}^2)\T]\T)$.	Let $ {N}_1, {N}_2 \in \mathbb{R}^{3\times 15} $ be the first three   rows  of matrices $\bar{C} \bar{A} $ and $\bar{C}\bar{A}^2$, respectively. One can show that  $\text{rank}({\mathcal{O}} )= \text{rank}(\bar{\mathcal{O}})$ with 
	\begin{align}
		\bar{\mathcal{O}} :=\begin{bmatrix}
			\bar{C}\T, 
			{N}_1\T, 
			{N}_2\T
		\end{bmatrix}\T  
		. \label{eqn:O1}
	\end{align}   	
	Since there exist at least three  non-aligned landmarks, for the sake of simplicity, we assume that the first three landmarks are not aligned. Define $u_i = [u_{i1}, u_{i2}, u_{i3}]\T: = p_1 - p_{i+1}$ for $i=1,2$. One can easily show that vectors $u_1, u_2$ and $g$ are linearly independent, since the plane of these three landmarks is not parallel to the gravity vector. Let $N_0$ be the first nine rows of $\bar{C}$. Applying the matrix row and column operations on matrix $ [N_0\T, N_1\T, N_2]\T$, we obtain the following matrix:
	\begin{align} 
		\bar{\mathcal{O}}'  
		= \begin{bmatrix}
			I_3  & -p_{11}I_3 & -p_{12}I_3 & -p_{13}I_3 & 0_{3}    \\
			0_{3}  &  u_{11}  I_3   & u_{12} I_3  & u_{13} I_3    & 0_{3}     \\
			0_{3}   & u_{21}I_3   & u_{22}I_3  & u_{23}I_3    & 0_{3}        \\ 		
			0_{3}  & 0_{3} & 0_{3}  & 0_{3}   & I_3  \\
			0_{3}  & g_1 I_3        & g_2 I_3  & g_3 I_3   & 0_{3}  
		\end{bmatrix}. \label{eqn:barO'}
	\end{align} 
	One can show that $\bar{\mathcal{O}}'$ has full rank of $15$ since $u_1, u_2$ and $g$ are linearly independent. Then, it is straightforward to verify that   $  \text{rank}(\mathcal{O}) = \text{rank}(\bar{\mathcal{O}}) \geq \text{rank}(\bar{\mathcal{O}}') = 15$, which implies that the  pair $(\bar{A}, \bar{C})$  is uniformly observable, and there exist   constants  $\delta,\mu'>0$ such that 
	$  
	W'_o(t,t+\delta)  \geq \mu' I_{15} 
	$  
	for all $t\geq 0$. Choosing   $0<\mu< \bar{\epsilon}   \mu'$, from \eqref{eqn:W_oA_C2}, it follows that 
	$
	W_o(t,t+\delta)  
	\geq  T(t)	W'_o(t,t+\delta') T^{-1}(t) \geq  \bar{\epsilon} \mu' T(t)  T^{-1}(t)   \geq \mu I_{15}
	$, \ie, the pair $(A(t),C(t))$ is uniformly observable.
	This completes the proof.

	\subsection{Proof of Lemma \ref{lemma:AC3}}\label{sec:lemmaAC3}  
	Following similar steps as in the proof of  Lemma \ref{lemma:AC2},  from \eqref{eqn:gramiancondition} and \eqref{eqn:def_T}-\eqref{eqn:def_barT}, one can show that
	\begin{align}
		W_o(t,t+\delta) 
		&=\frac{1}{\delta} \int_{t}^{t+\delta} \Phi(\tau,t)\T C\T(\tau)   C(\tau)  \Phi(\tau,t) d\tau   \nonumber \\ 	 
		&  =    T(t) W'_o(t,t+\delta)   T\T(t)   \label{eqn:W_oA_C3}
	\end{align}
	where $ W'_o(t,t+\delta) =\frac{1}{\delta}\int_{t}^{t+\delta} \bar{\Phi}\T(\tau,t)    \bar{C}\T \bar{\Theta}\T(\tau)   \bar{\Theta}(\tau)  \bar{C}   \bar{\Phi}(\tau,t) \\ d\tau$ with $\bar{\Theta}(t) = \bar{T}\T(t)\Theta(t)\bar{T}(t) =  \blkdiag(\bar{\Pi}_1,\dots, \bar{\Pi}_N)$ and $\bar{\Pi}_i = R\Pi_i R\T=I_3 - \frac{p_i-p'}{\|p_i-p'\|}$ with $p'=p+Rp_c$ for each $i=1,2,\dots,N$.  
	
	Next, we are going to show that the pair $(\bar{A},\bar{\Theta}(t)  \bar{C})$ is  uniformly observable. Since $\bar{A}$ is nilpotent with $\bar{A}^3 = 0$, from Lemma \ref{lemma:lemmaUOC}, the pair $(\bar{A},\bar{\Theta}(t)  \bar{C})$ is uniformly observable if there exist scalars $\delta,\mu>0$ such that inequality \eqref{eqn:gramiancondition1} holds 
	with $\mathcal{O}(t)  = [(\bar{\Theta}(t)\bar{C})\T, (\bar{\Theta}(t)\bar{C}\bar{A})\T, (\bar{\Theta}(t)\bar{C}\bar{A}^2)\T]\T$. To verify the above condition,  we proceed by contradiction. Assume that 
	\begin{align}
		\forall \bar{\delta},\bar{\mu}>0, \exists t\geq 0, \min_{z\in \mathbb{S}^{n-1}}\frac{1}{\bar{\delta}}  \int_{t}^{t+\bar{\delta}}\| \mathcal{O}(\tau)z\|^2 d\tau  < \bar{\mu}.  \label{eqn:O_contradiction}
	\end{align}
	Similar to the arguments in the proof of Lemma \eqref{lemma:lemmaUOC}, consider a sequence $\{\mu_q\}_{q\in \mathbb{N}}$ of positive numbers converging to zero, and an arbitrary positive scalar $\bar{\delta}$. Then, there must exist a sequence of time instants $\{t_q\}_{q\in \mathbb{N}}$ and a sequence of   vectors $\{z_q\}_{q\in \mathbb{N}}$ with $z_q\in \mathbb{S}^{n-1}$ such that for any $q\in \mathbb{N}$ one has $  \int_{t_q}^{t_q+\bar{\delta}} \|\mathcal{O}(\tau)z_q\|^2 d\tau < \bar{\mu}_q$. Since $\mathbb{S}^{n-1}$ is compact, there exists a sub-sequence of $\{z_q\}_{q\in \mathbb{N}}$ which converges to a limit $\bar{z}\in \mathbb{S}^{n-1}$. Therefore, since $\Theta(t)$ is bounded, one has  
	\begin{align}
		\lim_{q\to  \infty}  \int_{0}^{  \bar{\delta}} \|\mathcal{O}(t_q+ \tau) \bar{z}\|^2 d\tau=0   \label{eqn:contradictionQ} 
	\end{align}  
	by a change of integration variable. 
	Consider the following technical result  whose proof is given at the end of this proof. 
	\begin{lemma}\label{lemma:Q=0} 
		Let $\bar{\mathcal{O}}(t) = [(\bar{\Theta}(t)\bar{C})\T,N_1\T,N_2\T]\T$. From \eqref{eqn:contradictionQ}, one has
		\begin{equation}
			\lim\nolimits_{q\to  \infty}   \|\bar{\mathcal{O}}( t_q +s) \bar{z} \|^2    =0, \forall  s\in [0,\bar{\delta}]. \label{eqn:contradictionQ2}
		\end{equation} 
	\end{lemma}		
	Next, we consider the following cases  	
	\begin{itemize}
		\item [i)] \textit{Camera in motion:}	
		Let $\bar{C}_i = [I_3, p_{i1}I_3, p_{i2}I_3, p_{i3}I_3, 0_3]$ associated to $i$-th landmark. Note that  $\ell_1,\ell_2,\ell_3$ are not-aligned and their plane is not parallel to the gravity vector. Inequality $\bar{C} \bar{z}\neq 0_{3N\times 1}$ implies that there exists at least one of  $\bar{C}_i \bar{z}, i =\ell_1,\ell_2,\ell_3$, which is different from zero. Without loss of generality, let $\bar{z}'':=\bar{C}_i \bar{z}  \neq 0$ with $i \in \{\ell_1,\ell_2,\ell_3\}$. From \eqref{eqn:contradictionQ2}, one obtains
		\begin{align}
			\lim\nolimits_{q\to  \infty}  \|\bar{\Pi}_{i}(t_q + s) \bar{C}_i \bar{z} \|^2   =  0, \quad s \in [0,\bar{\delta}].  \label{eqn:PiC1z}
		\end{align}
		For the sake of simplicity, let $u(t) := R(t)R_c y_i(t) \in \mathbb{S}^2$. Using the facts $\bar{\Pi}_i = \pi(u)$   and  $\|\pi(u) y\|^2 = y\T \pi(u) y = - y\T (u^\times)^2 y = \|x\times y\|^2$, for any $u\in \mathbb{S}^2, y\in \mathbb{R}^3$, it follows from \eqref{eqn:PiC1z} that  
		$
		\lim\nolimits_{q\to  \infty} \| u(t_q +s ) \times \bar{z}'' \|^2    = 0, \forall s \in [0,\bar{\delta}].
		$ This implies that for any $\bar{\mu}'$, there exists $q^*$ such that for all $q\geq q^*$
		\begin{align}
			\| u(t_q+s)\times \bar{z}'' \|^2 < \bar{\mu}' ,~~ \forall   s\in [0,\bar{\delta}].  
		\end{align}  
		Motivated by the proof of \cite[Lemma 4]{berkane2020nonlinear}, let $u_1 = u(t_q)$ and $u_2 = u(t_q+\bar{\delta})$, and choose  $\bar{\mu}' = (\epsilon \|\bar{z}''\|)^2/(4+2\epsilon^2)$ such that $\| u_1  \times \bar{z}'' \|^2 + \|  u_2  \times \bar{z}'' \|^2 < (\epsilon \|\bar{z}''\|)^2/ (2+\epsilon^2)$.   The case where $\|u_1 \times u_2\| = 0$ is trivial. Let $\|u_1 \times u_2\| \neq 0$  and $\bar{z}'' = \alpha_1 u_1 + \alpha_2 u_2 + \alpha_3 u_1 \times u_2$ with constants $\alpha_i\in \mathbb{R},i=1,2,3$. Then, one has 
		$
		\| u_1 \times \bar{z}'' \|^2 + | u_2 \times \bar{z}'' \|^2  
		=  (\alpha_1^2 + \alpha_2^2 + 2 \alpha_3^2 ) \|u_1 \times u_2\|^2,  
		$
		where we made use of the fact  $\| u_i \times \bar{z}'' \|^2 = \alpha_i^2 \|u_1 \times u_2\| + \alpha_3^2 \|u_1 \times u_2\|^2$ for $i=1,2$.			
		One  can also show $\|\bar{z}''\|^2   \leq 2\alpha_1^2 + 2\alpha_2^2 +   \alpha_3^2 \|u_1 \times u_2\|^2  \leq  (2+\|u_1 \times u_2\|^2)(\alpha_1^2 + \alpha_2^2 +   \alpha_3^2 )$. Hence, one obtains 
		\begin{align*}
			\frac{ \|u_1 \times u_2\|^2 }{ 2+\|u_1 \times u_2\|^2 }  &\leq  \frac{  \|u_1 \times \bar{z}'' \|^2 + | u_2 \times \bar{z}'' \|^2 }{\|\bar{z}''\|^2}    < \frac{\epsilon^2}{2+\epsilon^2}.
		\end{align*}
		Since the function $ f(x)= {x^2}/{(2+x^2)}$ is  monotonically increasing (\ie, $\partial f/\partial x >0$) for all $x>0$, one obtains
		$
		\|u(t_q)\times u(t_q+\bar{\delta})\| = \|u_1 \times u_2\| < \epsilon
		$
		. From the definition  of $u(t)$ and the fact that $\bar{\delta}$ can be arbitrary large, this contradicts item (i) of the Lemma that for any $t^*\geq 0$, there exists $t>  t^*$ such that $\|u(t) \times u(t^*)\| > \epsilon$.

		\item[ii)] \textit{Motionless Camera:}
		In this case the matrix $\bar{\mathcal{O}}$ in  \eqref{eqn:contradictionQ2}  is constant.	For any $u \in \mathbb{R}^{3}$, one   shows   $  \bar{\Pi}_i  u      = u + (p_i-p') \bar{z}'_i$ with   $\| p_i-p' \|\neq 0$ and $\bar{z}_i'=-  {(p_i-p')\T u }/{\| p_i-p' \|^2} \in \mathbb{R} $ for all $i=1,2,\dots,N$. Note that identity $p'=p_i$ implies that the camera location coincides with the $i$-th landmark, which does not hold since all the landmarks are measurable by assumption. Let $\bar{z}'=[\bar{z}_1',\dots, \bar{z}_N']\T$ and $z'':=[\bar{z}\T,(\bar{z}')\T]\T$, which is non-zero  since $\bar{z}$ is non-zero. Hence, identity \eqref{eqn:contradictionQ2} can be rewritten as
		\begin{equation}
			\lim\nolimits_{q\to  \infty}  \left\| {\mathcal{O}}' z'' \right\|^2  =0 \label{eqn:contradictionQ3}
		\end{equation}
		with $\mathcal{O}'$ defined in \eqref{eqn:Motionless},
		which contradicts  item (ii) of the Lemma that  $\mathcal{O}'$  has full rank of $15+N$, \ie, $  \| {\mathcal{O}}' z''  \|^2  > 0, \forall t\geq 0$. 
	\end{itemize}
	Therefore, the assumption \eqref{eqn:O_contradiction} does not hold. Hence, for both cases, the pair $(\bar{A},C(t))$ is uniformly observable. It follows from \eqref{eqn:W_oA_C3} that the pair $({A}(t),C(t))$ is also uniformly observable.

	It remains to prove Lemma \ref{lemma:Q=0}. Since the velocity of the camera $v'=\frac{d}{dt}p'(t)$ is bounded,  the derivative of $\mathcal{O}(t_p+\tau)\bar{z}$ is bounded for all $q\in \mathbb{N}, \tau \in [0,\bar{\delta}]$, and then it follows that $\|\mathcal{O}(t_p+\tau)\bar{z}\|^2 $ is uniformly continuous. Since every uniformly continuous function is also Cauchy-continuous, it implies that $\|\mathcal{O}(t_p+\tau)\bar{z}\|^2$ is a Cauchy sequence of continuous functions and $\lim_{q \to \infty}\|\mathcal{O}(t_p+\tau)\bar{z}\|^2$ exists. Applying Lebesgue theorem, one has 
	$
	\lim_{q\to  \infty}  \int_{0}^{  \bar{\delta}} \|\mathcal{O}(t_q+ \tau) \bar{z}\|^2   =  \int_{0}^{  \bar{\delta}} \lim_{q\to  \infty} \|\mathcal{O}(t_q+ \tau) \bar{z}\|^2 =0. $
	Again, making use of the fact that $\lim_{q\to  \infty} \|\mathcal{O}(t_q+ \tau) \bar{z}\|^2  $ is uniformly continuous and non-negative, one can show that
	\begin{equation}
		\lim_{q\to  \infty} \|\mathcal{O}(t_q+ s) \bar{z}\|^2 = 0,\quad  \forall s\in [0,\bar{\delta}] \label{eqn:contradictionQ1}.
	\end{equation}
	From the definition of matrix $\mathcal{O}(t)$ and the identity \eqref{eqn:contradictionQ1}, it follows that $\sum_{i=0}^2\lim_{q\to  \infty} \| \bar{\Theta}(t_q+s)\bar{C} \bar{A}^i \bar{z}\|^2  = 0, \forall s\in [0,\bar{\delta}]$. Next, we are going to show that  $\bar{C}\bar{A}^2 \bar{z} = \bar{C}\bar{A} \bar{z} = 0_{3N \times 1}$. 
	Let $\bar{z} = [\bar{z}_1\T,\bar{z}_2\T,\dots, \bar{z}_5\T]\T \in \mathbb{R}^{15}$ with $\bar{z}_i \in \mathbb{R}^3$ for all $i=1,\dots,5$. Making use of the fact
	$
	\bar{\Theta}(t)\bar{C}\bar{A}^2 \bar{z}=  \bar{\Theta}(t) \begin{bmatrix}
		(g_1 \bar{z}_2  + g_2 \bar{z}_3 +  g_3 \bar{z}_4)\T,
		\dots, 
		(g_1 \bar{z}_2  + g_2 \bar{z}_3 +  g_3 \bar{z}_4)\T
	\end{bmatrix}\T $, 
	one can show that, if  $g_1 \bar{z}_2  + g_2 \bar{z}_3 +  g_3 \bar{z}_4 \neq 0$, the only solution of $\lim_{q\to  \infty}  \|\bar{\Theta}(t_q + s)\bar{C}\bar{A}^2 \bar{z}\| = 0$ is when the constant vector $g_1 \bar{z}_2  + g_2 \bar{z}_3 +  g_3 \bar{z}_4$ is collinear with $p_i-p'(t_q + s)$ for all $i=1, \dots,N$. This does no hold since there are at least three non-aligned landmarks by assumption. Hence, identity $\lim_{q\to  \infty} \| \bar{\Theta}(t_q+s)\bar{C}\bar{A}^2 \bar{z}\|^2  = 0$ implies $\bar{C}\bar{A}^2 \bar{z} = 0_{3N \times 1}$. Moreover,  making use of the fact
	$
	\bar{\Theta}(t)\bar{C}\bar{A}  z=    \begin{bmatrix}
		(\bar{\Pi}_1\bar{z}_5)\T, 
		\dots,
		(\bar{\Pi}_N\bar{z}_5)\T
	\end{bmatrix}\T $, 
	one can show that, if $\bar{z}_5\neq  0_3$, the only solution of $\lim_{q\to  \infty}  \|\bar{\Theta}(t_q + s)\bar{C}\bar{A}  \bar{z}\| = 0$ is when the constant vector $\bar{z}_5$ is collinear with $p_i-p'(t_q + s)$ for all $i=1, \dots,N$. This does no hold since there are at least three non-aligned landmarks by assumption. These identities are not satisfied since there are at least three non-aligned landmarks by assumption. Hence, $\lim_{q\to  \infty} \| \bar{\Theta}(t_q+s)\bar{C}\bar{A}  \bar{z}\|^2   = 0$ implies $\bar{C}\bar{A}  \bar{z} = 0_{3N\times 1}$.	From \eqref{eqn:O1} and the facts $\bar{C}\bar{A}^2 \bar{z} = \bar{C}\bar{A} \bar{z} = 0_{3N \times 1}$, one has $\bar{C}\bar{z} \neq  0_{3N\times 1} $. Then, from the definition of $\bar{\mathcal{O}}$, identity \eqref{eqn:contradictionQ1} can be reduced to \eqref{eqn:contradictionQ2}.  
	This completes the proof.

	\subsection{Proof of Proposition \ref{pro:non-observable}} \label{sec:non-observable}
	\begin{figure*}[htp] 
		\begin{align} 
			\mathcal{O}''&=\left[\begin{array}{ccccc|cccccc}
				I_3 & -p_{11}  I_3 &-p_{12} I_3 & -p_{13}I_3 & 0_3 & p_1-p' & 0_{3\times 1} & 0_{3\times 1}    & 0_{3\times 1}  & \dots  & 0_{3\times 1} \\
				0_3 & u_{11}  I_3 & u_{12}  I_3 & u_{13}  I_3 & 0_3 & -u_1 & p_2-p' & 0_{3\times 1} & 0_{3\times 1} &  \dots  & 0_{3\times 1} \\
				0_3 & u_{21}  I_3 & u_{22}  I_3 & u_{23}  I_3  & 0_3 & -u_2 &  0_{3\times 1} &  p_3-p'    & 0_{3\times 1}  &  \dots    & 0_{3\times 1}\\ 
				0_3 & 0_3 & 0_3 & 0_3 &I_3 & 0_{3\times 1}  & 0_{3\times 1} & 0_{3\times 1} &   0_{3\times 1}  &  \dots  & 0_{3\times 1}\\ 
				0_3 & g_1 I_3 & g_2 I_3 & g_3 I_3 & 0_3 & 0_{3\times 1} &  0_{3\times 1} & 0_{3\times 1}  &   0_{3\times 1}  &  \dots & 0_{3\times 1}\\		
				\hline	
				0_3 & 0_3 & 0_3 &0_3 & 0_3 &\alpha_{43} g &\alpha_{41}(p_2-p') & \alpha_{42}(p_3-p')&p_4-p'&  \dots  & 0_{3\times 1}\\	
				\vdots & \vdots & \vdots & \vdots & \vdots & \vdots  & \vdots & \vdots & \vdots &\ddots& \vdots\\			
				0_3 & 0_3 & 0_3 &0_3 & 0_3 & \alpha_{N3} g & \alpha_{N1}(p_2-p') &  \alpha_{N2}(p_3-p')& 0_{3\times 1}  & \dots &  p_N-p'    
			\end{array} \right]   \label{eqn:O''} 
		\end{align}  
		\hrulefill 
	\end{figure*} 
	Consider the matrix  $\mathcal{O}'$  defined in \eqref{eqn:Motionless} with $N\geq 5$. Since all the landmarks are not aligned, for the sake of simplicity, we assume that the first three landmarks are not aligned. Let $u_i = [u_{i1}, u_{i2}, u_{i3}]\T: = p_1 - p_{i+1}$ for $i=1,2$. Hence, for each landmark $p_i$,  there exist scalars $\alpha_{ij},  j=1,2,3$ such that  $p_i-p_1 = \sum_{j=1}^2 \alpha_{ij} u_j + \alpha_{i3} g$. Note that if all the landmarks are located in the same plane, one has $\alpha_{i3} = 0$ for all $ i=4,\dots,N$. Otherwise, one can always find three non-aligned landmarks, denoted by $p_1, p_2$ and $p_3$,  whose plane is not parallel to the gravity vector, \ie,  $u_1,u_2$ and $g$ are linearly independent. Applying similar column and row operations as in \eqref{eqn:barO'} on the matrix $\mathcal{O}'$, one obtains $\mathcal{O}''$ in \eqref{eqn:O''} (\ie, $\text{rank}(\mathcal{O}') = \text{rank}(\mathcal{O}'')$), which can be rewritten in the form of a block upper triangular matrix  as
	\begin{align*}
		\mathcal{O}'' = \begin{bmatrix}
			\bar{\mathcal{O}}' & \Delta \\
			0_{(3N-9) \times 15} & M'
		\end{bmatrix}
	\end{align*} 
	where $\bar{\mathcal{O}}'\in \mathbb{R}^{15\times 15}$ is equivalent to the one defined in \eqref{eqn:barO'}.  
	It is obvious to show that matrix $\mathcal{O}''$ does not have full rank if all the landmarks are located in the plane parallel to the gravity vector (\ie, $\text{rank}(\mathcal{O}')< 15$).

	Now, we assume that all the landmarks are not located in the same plane, and there exist three non-aligned landmarks, denoted by $p_1, p_2$ and $p_3$, whose plane is not parallel to the gravity. It follows that matrix $\mathcal{O}'$ has full rank and is invertible. Then, one can show that any non-zero column of $\Delta$ is linearly dependent on the columns of $\bar{\mathcal{O}}'$. Hence, $\mathcal{O}''$ has  full rank of $15+N$ if the matrix $M'$  has rank of $N$.  Since all the landmarks are measurable by assumption, one has $p'-p_i \neq 0$ for all $i=1,2,\dots,N$ and the last $N-3$ columns of $M'$ are linearly independent. Applying  the column and row operations, one can show that  matrix $M'$  has the same rank as matrix $M''$ defined as follows: 
	\begin{align}  \arraycolsep=1.0pt\def\arraystretch{1.1}
		\small M''  = \begin{bmatrix} 
			\alpha_{43}'\breve{p}_{14} &\alpha_{41}\breve{p}_{24}& \alpha_{42}\breve{p}_{34}&p_4-p'& \dots & 0_{3\times 1}\\	
			\vdots  & \vdots & \vdots & \vdots &\ddots& \vdots\\			
			\alpha_{N3}' \breve{p}_{1N}& \alpha_{N1}\breve{p}_{2N}&  \alpha_{N3}\breve{p}_{3N}& 0_{3\times 1}  & \dots &  p_N-p' 
		\end{bmatrix}  \label{eqn:M''}
	\end{align}  
	where $\breve{p}_{ij}:=p_i-p_j, \forall i,j =1,\dots,N$, and we made use of the facts $\alpha_{i3} g -\alpha_{i1}\breve{p}_{2i} -\alpha_{i2}\breve{p}_{3i}   =  \sum_{j=1}^2 \alpha_{ij} u_j + \alpha_{i3} g -  (\alpha_{i1} + \alpha_{i2})\breve{p}_{1i} =  \alpha_{i3}'\breve{p}_{1i}  $ with $\alpha_{i3}'  = -(1+ \alpha_{i1} + \alpha_{i2})$ for all $i=4,\dots,N$. 
	\begin{itemize}
		\item [(a)] Consider the case where all the landmarks are located in the same plane. It implies that $\alpha_{i3}=0$ for all $i=4,\dots,N$. From \eqref{eqn:M''}, one can   show that the  first three columns of $M''$ are linearly dependent, \ie, 
		$\alpha_{i3}'\breve{p}_{1i} + \alpha_{i1}\breve{p}_{2i} + \alpha_{i2}\breve{p}_{3i} = \alpha_{i3} g = 0_{3\times 1}$ for each $i=4,\dots,N$. Hence, the rank of $M''$ is less than $N$.
		
		\item [(b)] Consider the case where one of the first three columns is zero. 
		\begin{itemize}
			\item If $\alpha_{i1}=0$ for all $i=4,\dots,N$, one has $\breve{p}_{i1} = -\alpha_{i2}(p_3-p_1)+ \alpha_{i3} g$. It follows that all the landmarks are located in the plane that contains $p_1, p_3$ and is parallel to the gravity vector.
			\item If $\alpha_{i2}=0$ for all $i=4,\dots,N$, one has $\breve{p}_{i1} = -\alpha_{i1}(p_2-p_1)+ \alpha_{i3} g$. It follows that all the landmarks are located in the plane that contains $p_1, p_2$ and is parallel to the gravity vector.		
			\item If $\alpha_{i3}'=0$ (\ie, $\alpha_{i1}+ \alpha_{i2} = -1$) for all $i=4,\dots,N$, one has    $\breve{p}_{i1} =  -\alpha_{i1}(p_2-p_1) -\alpha_{i2}(p_3-p_1)+ \alpha_{i3} g$, which can be rewritten as $\breve{p}_{i3}= -\alpha_{i1}(p_2-p_3)+ \alpha_{i3} g$. It follows that all the landmarks are located in the plane that contains $p_2, p_3$ and is parallel to the gravity vector.
		\end{itemize}
		Hence, the rank of $M'$ is less than $N$ if landmarks $p_4,\dots,p_N$ are located in the same plane that contains   two  of landmarks $p_1, p_2, p_3$ and is parallel to the gravity vector.
		
		\item [(c)] Consider the case where landmarks $p_4,\dots,p_N$ are aligned with one of the landmarks $p_1, p_2, p_3$ and the camera position $p'$. Without loss of generality, let $p_1, p_4,\dots,p_N$ and $p'$ be aligned. Then, one can show that the first column of the matrix $M''$  is linearly dependent on its last $N-3$ columns, which implies that the rank of $M'$ is less than $N$.
		
		\item [(d)] Consider the case where there exists an index $4<\ell<N$ such that landmarks $p_4,\dots,p_{\ell}$ are located in the same plane that contains two of the landmarks $p_1,p_2,p_3$ (for example, $p_1$ and $p_2$) and is parallel to the gravity vector, and landmarks $p_{\ell+1},\dots,p_N$ are aligned with the third landmark (\ie, $p_3$) and the camera position $p'$. It follows that $\alpha_{i2} = 0$ for all $  i =4, \dots, \ell$,  $p_3-p_i$   and   $p_i-p'$ are collinear for all $i=\ell +1, \dots, N$, which implies that the rank of $M'$ is less than $N$.
	\end{itemize}
	If none of these cases hold, one can conclude that matrix $\mathcal{O}'$ in \eqref{eqn:Motionless}  has full rank. This completes the proof.

	\bibliographystyle{IEEEtran}
	\bibliography{mybib}

\end{document}